\documentclass{amsart}

\usepackage{amsmath, amssymb, amsthm}
\usepackage{bbm}
\usepackage{todonotes}
\usepackage{hyperref}
\usepackage{array}

\newcounter{claimo}[section]

\theoremstyle{plain}
\newtheorem{theorem}{Theorem}[section]
\newtheorem{lemma}[theorem]{Lemma}
\newtheorem{proposition}[theorem]{Proposition}
\newtheorem{corollary}[theorem]{Corollary}
\newtheorem{claim}[claimo]{Claim}
\newtheorem*{claim*}{Claim}

\newtheorem{fact}[theorem]{Fact}

\theoremstyle{definition}
\newtheorem{definition}[theorem]{Definition}

\newcommand{\dom}{\mathrm{dom}}
\newcommand{\bb}{\mathbb}

\newcommand{\otp}{\mathrm{otp}}
\newcommand{\cof}{\mathrm{cof}}

\newcommand{\acc}{\mathrm{acc}}
\newcommand{\pred}{\mathrm{pred}}

\newcommand{\betrag}[1]{\vert{#1}\vert}
\newcommand{\height}[2]{{\rm{ht}}_{{#2}}(#1)}
\newcommand{\lub}{{\rm{lub}}}

\newcommand{\ran}[1]{{{\rm{ran}}(#1)}}

\newcommand{\map}[3]{{#1}:{#2}\longrightarrow{#3}}
\newcommand{\Map}[5]{{#1}:{#2}\longrightarrow{#3};~{#4}\longmapsto{#5}}
\newcommand{\pmap}[4]{{#1}:{#2}\xrightarrow{#4}{#3}}
\newcommand{\Set}[2]{\{{#1}~\vert~{#2}\}}
\newcommand{\seq}[2]{\langle{#1}~\vert~{#2}\rangle}

\newcommand{\anf}[1]{{\text{``}\hspace{0.3ex}{#1}\hspace{0.3ex}\text{''}}}

\newcommand{\Poti}[2]{{\mathcal{P}}_{#2}(#1)}
\newcommand{\HH}[1]{{\rm{H}}(#1)}

\newcommand{\Add}[2]{{\rm{Add}}({#1},{#2})}
\newcommand{\Col}[2]{{\rm{Col}}({#1},{#2})}

\newcommand{\On}{{\rm{On}}}
\newcommand{\LL}{{\rm{L}}}

\newcommand{\ZFC}{{\rm{ZFC}}}

\newcommand{\DDD}{{\mathbb{D}}}

\newcommand{\PPP}{{\mathbb{P}}}
\newcommand{\QQQ}{{\mathbb{Q}}}
\newcommand{\RRR}{{\mathbb{R}}}
\newcommand{\SSS}{{\mathbb{S}}}
\newcommand{\TTT}{{\mathbb{T}}}
\newcommand{\UUU}{{\mathbb{U}}}

\newcommand{\VV}{{\rm{V}}}

\newcommand{\calC}{\mathcal{C}}

\begin{document}
\title{Squares, ascent paths, and chain conditions}

\author{Chris Lambie-Hanson} 
\address{Department of Mathematics, Bar-Ilan University, Ramat-Gan 5290002, Israel}
\urladdr{http://math.biu.ac.il/~lambiec/}

\author{Philipp L\"ucke}
\address{Mathematisches Institut, Universit\"at Bonn,
Endenicher Allee 60, 53115 Bonn, Germany}
\urladdr{http://www.math.uni-bonn.de/people/pluecke/}



\thanks{During the preparation of this paper, the second author was partially supported by the Deutsche Forschungsgemeinschaft under the grant LU2020/1-1. The initial results of this paper were obtained while the authors were participating in the \emph{Intensive Research Program on Large Cardinals and Strong Logics} at the Centre de Recerca Matem\`{a}tica in Barcelona during the fall of 2016. The authors would like to thank the organizers for the opportunity to participate in the program. 
Further results were obtained while the first author was visiting the second author in Bonn during the spring of 2017. The first author would like to thank the Deutsche Forschungsgemeinschaft for the financial support of this visit through the above grant.}

\begin{abstract}
 With the help of various square principles, we obtain results concerning the consistency strength of 
 several statements about trees containing ascent paths, special trees, and strong chain conditions.  
 Building on a result that shows that Todor\v{c}evi\'{c}'s principle $\square(\kappa)$ implies an indexed 
 version of $\square(\kappa,\lambda)$, we show that for all infinite, regular cardinals $\lambda<\kappa$, 
 the principle $\square(\kappa)$ implies the existence of a $\kappa$-Aronszajn tree containing a $\lambda$-ascent path. 
 We then provide a complete picture of the consistency strengths of statements relating the interactions of 
 trees with ascent paths and special trees. As a part of this analysis, we construct a model of set theory in which 
 $\aleph_2$-Aronszajn trees exist and all such trees contain $\aleph_0$-ascent paths.  Finally, we use our techniques 
 to show that the assumption that the $\kappa$-Knaster property is countably productive and the assumption that every 
 $\kappa$-Knaster partial order is $\kappa$-stationarily layered both imply the failure of $\square(\kappa)$. 
\end{abstract}

\maketitle


\section{Introduction}\label{section:Introduction}

The existence or non-existence of cofinal branches is one of the most fundamental properties of set-theoretic trees\footnote{A summary of basic definitions concerning set-theoretic trees can be found in Section \ref{BasicDefinitions}.} of uncountable regular height. Important examples of trees without cofinal branches are given by \emph{special trees}. Given an infinite cardinal $\mu$, a tree of height $\mu^+$ is special if it can be decomposed into $\mu$-many antichains. This notion was generalized by Todor\v{c}evi\'{c} to the class of all trees of uncountable regular heights (see Definition \ref{definition:SpacialTree}). It is easy to see that a special tree does not contain a cofinal branch, not only in the ground model $\VV$, but also in all outer models of $\VV$ in which its height remains a regular cardinal.

In contrast, it is possible to use the concept of \emph{ascent paths}, introduced by Laver, to obtain interesting examples of branchless, non-special trees of uncountable regular height. Given infinite regular cardinals $\lambda<\kappa$, a $\lambda$-ascent path through a tree $\TTT$ of height $\kappa$ is a sequence $\seq{\map{b_\alpha}{\lambda}{\TTT}}{\alpha<\kappa}$ of functions with the property that $b_\alpha(i)$ is contained in the $\alpha$-th level of $\TTT$ for all $\alpha<\kappa$ and $i<\lambda$ and, for all $\alpha<\beta<\kappa$, there is an $i<\lambda$ with $b_\alpha(j)<_\TTT b_\beta(j)$ for all $i\leq j<\lambda$.  A theorem of Shelah (see {\cite[Lemma 3]{MR964870}}) then shows that, if $\mu$ is an uncountable cardinal and $\lambda<\mu$ is a regular cardinal with $\lambda\neq\cof(\mu)$, then every tree of height $\mu^+$ that contains a $\lambda$-ascent path is not special. Note that this  shows that trees containing ascent paths are non-special in a very absolute way, because it implies that they remain non-special in every outer model of $\VV$ in which $\mu$ and $\mu^+$ remain cardinals and $\cof(\lambda)\neq\cof(\mu)$ holds. This result was later strengthened by Todor\v{c}evic and Torres P\'{e}rez in \cite{MR2965421} and by the second author in \cite{ascending_paths} (see Lemma \ref{lucke_lemma}).

Many authors have dealt with the construction of trees of various types containing ascending paths (see, for example, \cite{MR3596614}, \cite{MR1376756}, \cite{MR732661}, \cite{ChrisTreesSquareReflection}, \cite{ascending_paths},  \cite{MR964870} and \cite{MR929410}). In particular, the constructions of Shelah and Stanley in \cite{MR964870} and Todor\v{c}evi\'c in \cite{MR929410} show that, given infinite, regular cardinals $\lambda<\kappa$, the existence of a $\kappa$-Aronszajn tree containing a $\lambda$-ascent path follows from the existence of a $\square(\kappa)$-sequence that avoids\footnote{Given an uncountable regular cardinal $\kappa$, a $\square(\kappa)$-sequence $\seq{C_\alpha}{\alpha<\kappa}$ \emph{avoids} a stationary subset $S$ of $\kappa$ if $\acc(C_\alpha)\cap S=\emptyset$ holds for all $\alpha\in\acc(\kappa)$.} a stationary subset $S$ of $\kappa$ consisting of limit ordinals of cofinality $\lambda$ (see {\cite[Theorem 4.12]{ascending_paths}} and {\cite[Section 3]{MR3523658}}). 
Our first main result shows that such a tree can be constructed from a $\square(\kappa)$-sequence without additional properties. This answers {\cite[Questions 6.5 and 6.6]{ascending_paths}}.

\begin{theorem} \label{square_ascent_path_thm}
 Let $\lambda<\kappa$ be infinite, regular cardinals. If $\square(\kappa)$ holds, then there is a $\kappa$-Aronszajn tree with a $\lambda$-ascent path. 
\end{theorem}

It is easy to see that, if $\kappa$ is a weakly compact cardinal and $\TTT$ is tree of height $\kappa$ containing a $\lambda$-ascent path with $\lambda < \kappa$, then $\TTT$ contains a cofinal branch. Moreover, basic arguments, presented in  {\cite[Section 3]{ascending_paths}}, show that, if $\kappa$ is a weakly compact cardinal, $\mu<\kappa$ is a regular, uncountable cardinal, and $G$ is $\Col{\mu}{{<}\kappa}$-generic over $\VV$, then every tree of height $\kappa$ in $\VV[G]$ that contains a $\lambda$-ascent path with $\lambda<\mu$ already has a cofinal branch. Since seminal results of Jensen and Todor\v{c}evi\'{c} show that, for uncountable regular cardinals $\kappa$, a failure of $\square(\kappa)$ implies that $\kappa$ is weakly compact in G\"odel's constructible universe $\LL$ (see {\cite[Section 6]{jensen_fine_structure}} and {\cite[(1.10)]{MR908147}}), the above theorem directly yields the following corollary showing that the existence of regular cardinals $\lambda < \mu$ such that there are no $\mu^+$-Aronszajn trees with $\lambda$-ascent paths is equiconsistent with the existence of a weakly compact cardinal.

\begin{corollary} \label{square_ascent_path_cor}
 Let $\kappa$ be an uncountable regular cardinal. If there is an infinite regular cardinal $\lambda<\kappa$ with the property that there are no $\kappa$-Aronszajn trees with $\lambda$-ascent paths, then $\kappa$ is a weakly compact cardinal in $\LL$. \qed
\end{corollary}

Starting with the above theorem, we provide a complete picture of the consistency strengths of statements relating the interactions of trees with ascent paths and special trees. For concreteness, we will speak here about $\aleph_2$-Aronszajn trees and $\aleph_0$-ascent paths, 
but the same results will hold for $\mu^+$-Aronszajn trees and $\lambda$-ascent paths, provided $\lambda < \mu$ are infinite, 
regular cardinals. In what follows, if $\TTT$ is an $\aleph_2$-Aronszajn tree, then $S(\TTT)$ denotes the assertion that 
$\TTT$ is special and $A(\TTT)$ denotes the assertion that $\TTT$ has an $\aleph_0$-ascent path.
Table \ref{table:Implications_Special_Ascent_Succ_Reg} provides a complete picture of the precise consistency strengths of various assertions relating the existence 
of special trees and the existence of trees with ascent paths, 
where the background assumption is that there are $\aleph_2$-Aronszajn trees, and the quantification is over the set of 
$\aleph_2$-Aronszajn trees.

\begin{table}[h]
\medskip
  \begin{tabular}{|*{4}{>{\centering\arraybackslash}p{.20\linewidth}|}}
 \cline{2-4}
 \multicolumn{1}{c | }{} & $\forall\TTT.A(\TTT)$ & $\exists\TTT.A(\TTT)$ & $\forall\TTT.\neg A(\TTT)$ \\ \hline
  $\forall\TTT.S(\TTT)$   & $0=1$ & $0=1$ & Weakly compact \\ \hline 
  $\exists\TTT.S(\TTT)$   & $0=1$ & ZFC & Weakly compact \\ \hline
  $\forall\TTT.\neg S(\TTT)$   & Weakly compact & Mahlo & Weakly compact \\ \hline
 \end{tabular}
  \medskip
 
 \caption{The consistency strengths of interactions between trees with ascent paths and special trees.}\label{table:Implications_Special_Ascent_Succ_Reg}
\end{table}

Besides Theorem \ref{square_ascent_path_thm}, the following result is the other main new ingredient in the determination of the consistency strengths in Table \ref{table:Implications_Special_Ascent_Succ_Reg}. The results of Section \ref{section:Forcing_preliminaries} will show that, for successors of regular cardinals and inaccessible cardinals, the consistency of the hypotheses of this theorem can be established from a weakly compact cardinal. Given an infinite, regular cardinal $\kappa$, we let $\Add{\kappa}{1}$ denote the partial order that adds a Cohen subset to $\kappa$. Moreover, given an uncountable, regular cardinal $\kappa$, we let $\mathrm{TP}(\kappa)$ denote the statement that the tree property holds at $\kappa$.

\begin{theorem}\label{all_ascent_path_thm}
 Let $\lambda < \kappa$ be infinite, regular cardinals such that $\kappa=\kappa^{{<}\kappa}$ and $\mathbbm{1}_{\Add{\kappa}{1}}\Vdash\mathrm{TP}(\check{\kappa})$.  Then the following statements hold in a cofinality-preserving forcing extension of the ground model: 
  \begin{enumerate}
    \item There are $\kappa$-Aronszajn trees. 
       
    \item Every $\kappa$-Aronszajn tree contains a $\lambda$-ascent path.
  \end{enumerate}
\end{theorem}

In the last part of this paper, we use the techniques developed in this paper to study chain conditions of partial orders. There is a close connection between ascent paths and the infinite productivity of chain conditions, given by the fact that a result of Baumgartner (see {\cite[Theorem 8.2]{MR823775}}) shows that for every tree $\TTT$ of uncountable regular height $\kappa$ without cofinal branches, the canonical partial order $\PPP(\TTT)$ that specializes $\TTT$ using finite partial functions $\pmap{f}{\TTT}{\omega}{part}$ (see Definition \ref{definition:SpecializiationForcing}) satisfies the $\kappa$-chain condition, and that every $\lambda$-ascent path $\seq{\map{b_\alpha}{\lambda}{\TTT}}{\alpha<\kappa}$ through $\TTT$ induces an antichain $\Set{p_\alpha}{\alpha<\kappa}$ in the full support product $\prod_{i<\lambda}\PPP(\TTT)$ with $\dom(p_\alpha(i))=\{b_\alpha(i)\}$ and $p_\alpha(i)(b_\alpha(i))=0$ for all $\alpha<\kappa$ and $i<\lambda$ (see {\cite[Section 2]{ascending_paths}} for more details on this connection).

Our first application deals with failures of the infinite productivity of the $\kappa$-Knaster property. Remember that, given an uncountable regular cardinal $\kappa$, a partial order $\PPP$ is \emph{$\kappa$-Knaster} if every collection of $\kappa$-many conditions in $\PPP$ contains a subcollection of cardinality $\kappa$ that consists of pairwise compatible conditions. This strengthening of the $\kappa$-chain condition is of great interest because of its product behavior. In particular, the product of two $\kappa$-Knaster partial orders is $\kappa$-Knaster and the product of a $\kappa$-Knaster partial order with a partial order satisfying the $\kappa$-chain condition again satisfies the $\kappa$-chain condition. An easy argument (see {\cite[Proposition 1.1]{MR3620068}}) shows that, if $\kappa$ is a weakly compact cardinal, then the class of $\kappa$-Knaster partial orders is closed under $\mu$-support products for all $\mu<\kappa$. A combination of {\cite[Theorem 1.13]{MR3620068}} with {\cite[Theorem 1.12]{ascending_paths}} shows that the question of whether, for uncountable regular cardinals $\kappa$, the countable productivity of the $\kappa$-Knaster is equivalent to the weak compactness of $\kappa$ is independent of the axioms of $\ZFC$. The construction in \cite{MR3620068}, producing a model of set theory in which this characterization of weak compactness fails, starts from a model of $\ZFC$ containing a weakly compact cardinal. The following result and its corollary show that this assumption is necessary.

\begin{theorem} \label{square_productivity_thm}
 Let $\kappa$ be an uncountable regular cardinal with the property that $\square(\kappa)$ holds. If $\lambda<\kappa$ is an infinite, regular cardinal, then there is a partial order $\PPP$ with the following properties: 
  \begin{enumerate}
   \item If $\mu<\lambda$ is a (possibly finite) cardinal with $\nu^\mu<\kappa$ for all $\nu<\kappa$, then $\PPP^\mu$ is $\kappa$-Knaster. 
   
   \item $\PPP^\lambda$ does not satisfy the $\kappa$-chain condition. 
  \end{enumerate} 
\end{theorem}

The results of \cite{MR3620068} and \cite{ascending_paths} leave open the question whether it is consistent that the $\kappa$-Knaster property is countably productive for accessible uncountable regular cardinals, like $\aleph_2$. It is easy to see that this productivity implies certain cardinal arithmetic statements. Namely, if $\lambda<\kappa$ are infinite, regular cardinals, $\nu<\kappa$ is a cardinal with $\nu^\lambda\geq\kappa$ and $\PPP$ is a partial order of cardinality $\nu$ containing an antichain of size $\nu$ (e.g. the lottery sum of $\nu$-many copies of \emph{Cohen forcing} $\Add{\omega}{1}$), then $\PPP$ is $\kappa$-Knaster and the full support product $\PPP^\lambda$ contains an antichain of size $\nu^\lambda$.

\begin{corollary}
 Let $\lambda < \kappa$ be infinite, regular cardinals. If the class of $\kappa$-Knaster partial orders is closed under $\lambda$-support products, then $\kappa$ is weakly compact in $\LL$ and $\nu^\lambda<\kappa$ for all $\nu<\kappa$. \qed
\end{corollary}

Our second application deals with a strengthening of the $\kappa$-Knaster property introduced by Cox in \cite{Cox_Layerings}. Given an uncountable regular cardinal $\kappa$, a partial order $\PPP$ is \emph{$\kappa$-stationarily layered} if the collection of all regular suborders of $\PPP$ of cardinality less than $\kappa$ is stationary\footnote{This definition refers to Jech's notion of stationarity in $\Poti{A}{\kappa}$: a subset of $\Poti{A}{\kappa}$ is stationary in $\Poti{A}{\kappa}$ if it meets every subset of $\Poti{A}{\kappa}$ which is $\subseteq$-continuous and cofinal in $\Poti{A}{\kappa}$.} in the collection $\Poti{\PPP}{\kappa}$ of all subsets of $\PPP$ of cardinality less than $\kappa$. In {\cite{Cox_Layerings}}, Cox shows that this property implies the $\kappa$-Knaster property. The main result of \cite{MR3620068} shows that an uncountable regular cardinal is weakly compact if and only if every partial order satisfying the $\kappa$-chain condition is $\kappa$-stationarily layered. Moreover, it is shown that the assumption that every $\kappa$-Knaster partial order is $\kappa$-stationarily layered implies that $\kappa$ is a Mahlo cardinal with the property that every stationary subset of $\kappa$ reflects. In particular, it follows that this assumption characterizes weak compactness in certain models of set theory. In contrast, it is also shown in \cite{MR3620068} that there is consistently a non-weakly compact cardinal $\kappa$ such that every $\kappa$-Knaster partial order is $\kappa$-stationarily layered. The model of set theory witnessing this consistency is again constructed assuming the existence of a weakly compact cardinal. The following result shows that this assumption is necessary, answering {\cite[Questions 7.1 and 7.2]{MR3620068}}. 

\begin{theorem}\label{KnasterLayered_thm}
 Let $\kappa$ be an uncountable, regular cardinal. If every $\kappa$-Knaster partial order is $\kappa$-stationarily layered, then $\square(\kappa)$ fails. 
\end{theorem}


\section{Trees and ascent paths}\label{BasicDefinitions}

In this short section, we recall some fundamental definitions and results dealing with trees, special trees and ascent paths.

\begin{definition}
  A partial order $\TTT$ is a \emph{tree} if, for all $t \in \TTT$, the set $$\pred_\TTT(t) ~  = ~  \Set{s \in \TTT}{s <_\TTT t}$$  
  is well-ordered by the relation $<_\TTT$.
\end{definition}

\begin{definition}
 Let $\TTT$ be a tree.
  \begin{enumerate}
    \item For all $t \in\TTT$, we let $\height{t}{\TTT}$ denotes the order type of $\langle\pred_\TTT(t), <_\TTT\rangle$. 
    
    \item For all $\alpha \in \On$, we set $\TTT(\alpha) = \Set{t \in \TTT}{\height{t}{\TTT} = \alpha}$. 
    
    \item We let $\height{\TTT}{}$ denote the least ordinal $\alpha$ such that $\TTT(\alpha) = \emptyset$.  This ordinal is referred to as the \emph{height} of $\TTT$. 
    
    \item If $S \subseteq\height{\TTT}{}$, then $\TTT\restriction S$ is the suborder of $\TTT$ whose underlying set is $\bigcup \Set{\TTT(\alpha)}{\alpha \in S}$. 
    
    \item A \emph{branch through $\TTT$} is a subset $B$ of $\TTT$ that is linearly ordered by 
      $\leq_{\TTT}$. A branch $B$ is \emph{cofinal} if the set $\Set{\height{t}{\TTT}}{t \in B}$ is cofinal 
      in $\height{\TTT}{}$. 
      
    \item If $\kappa$ is a regular cardinal, then $\TTT$ is a \emph{$\kappa$-tree} if $\height{\TTT}{} = \kappa$ 
      and $\TTT(\alpha)$ has cardinality less than $\kappa$ for all $\alpha < \kappa$. A \emph{$\kappa$-Aronszajn tree} is a $\kappa$-tree without cofinal branches.  
\end{enumerate}
\end{definition}

\begin{definition}[Todor\v{c}evi\'{c}, \cite{MR657114}]\label{definition:SpacialTree}
  Let $\kappa$ be an uncountable regular cardinal, let $\TTT$ be a tree of height $\kappa$, and let $S$ be a subset of $\kappa$. 
  \begin{enumerate}
    \item A map $\map{r}{\TTT \restriction S}{\TTT}$ is \emph{regressive} if $r(t) <_\TTT t$ holds for all $t \in \TTT \restriction S$ with  $\height{t}{\TTT} > 0$. 
    
    \item The subset $S$ is \emph{non-stationary with respect to $\TTT$} if there is a regressive map  $\map{r}{\TTT\restriction S }{\TTT}$ such that, for every $t \in \TTT$,      there is a $\theta_t < \kappa$ and a function $\map{c_t}{r^{{-}1}``\{t\}}{\theta_t}$   that is injective on $<_\TTT$-chains. 
    
    \item The tree $\TTT$ is \emph{special} if $\kappa$ is non-stationary with respect to $\TTT$. 
  \end{enumerate}
\end{definition}

A result of Todor\v{c}evi\'{c} (see {\cite[Theorem 14]{MR793235}}) shows that for successor cardinals, the above notion of special trees coincides with the notion mentioned in Section \ref{section:Introduction}. 
One of the reasons for interest in special trees is that they are branchless in a very absolute way.

\begin{fact} \label{non_stationary_fact}
  Suppose that $\kappa$ is an uncountable regular cardinal, $\TTT$ is a tree of height $\kappa$, 
  and there is a stationary $S \subseteq \kappa$ such that $S$ is non-stationary with respect to 
  $\TTT$. Then there are no cofinal branches through $\TTT$. \qed 
\end{fact}

\begin{definition} \label{ascent_path_def}
 Let $\lambda < \kappa$ be cardinals with $\kappa$ uncountable and regular, let $\TTT$ be  a tree of height $\kappa$, and let $\vec{b} = \seq{\map{b_\alpha}{\lambda}{\TTT(\alpha)}}{\alpha < \kappa}$ be a sequence of functions.
  \begin{enumerate}
    \item The sequence $\vec{b}$ is a \emph{$\lambda$-ascending path through $\TTT$} if, for all $\alpha < \beta < \kappa$,       there are $i,j < \lambda$ with $b_\alpha(i) <_\TTT b_\beta(j)$. 
    
    \item The sequence $\vec{b}$ is a \emph{$\lambda$-ascent path through $\TTT$} if, for all $\alpha < \beta < \kappa$, 
      there is an $i< \lambda$ such that $b_\alpha(j) <_\TTT b_\beta(j)$ holds for all $i \leq j < \lambda$.  
      
    \item Suppose that $\vec{b}$ is a $\lambda$-ascent path through $\TTT$. Given $I < \lambda$ and a cofinal subset $B$ of $\kappa$, the pair $\langle I, B\rangle$ is a \emph{true  cofinal branch} through $\vec{b}$ if the following statements hold:
      \begin{enumerate}
        \item \label{c3c} If $\alpha,\beta\in B$ with $\alpha < \beta$, then $b_\alpha(i) <_\TTT b_\beta(i)$ for all $I \leq i < \lambda$. 
        
        \item \label{c3d} If $\beta \in B$ and $\alpha < \beta$ with $b_\alpha(i) <_\TTT b_\beta(i)$ for           all $I \leq i < \lambda$, then $\alpha \in B$.
      \end{enumerate}
  \end{enumerate}
\end{definition}

A $\lambda$-ascent path through a tree $\TTT$ is clearly a $\lambda$-ascending path through $\TTT$. The notion 
of a $\lambda$-ascent path is due to Laver and grew out of his work on higher Souslin Hypotheses in \cite{MR603771}. Ascent paths and ascending paths can be seen as generalized cofinal branches and, like cofinal branches, they provide 
concrete obstructions to a tree being special. The best current result in this direction is due to the second 
author, building upon work of Shelah in  \cite{MR964870} and Todor\v{c}evic and Torres P\'{e}rez in \cite{MR2965421}.
Given an uncountable regular cardinal $\kappa$ and a cardinal $\lambda<\kappa$, we let $E^\kappa_{{>}\lambda}$ denote the set of all $\alpha\in\acc(\kappa)$ with $\cof(\alpha)>\lambda$. 
The sets $E^\kappa_{{\geq}\lambda}$, $E^\kappa_\lambda$, etc. are defined analogously.

\begin{lemma}[{\cite[Lemma 1.6]{ascending_paths}}] \label{lucke_lemma}
  Let $\lambda < \kappa$ be infinite cardinals with $\kappa$ uncountable and regular, let $\TTT$ be a tree of height $\kappa$ and let $S \subseteq S^\kappa_{{>}\lambda}$ be stationary in $\kappa$. If $\kappa$ is not the successor 
  of a cardinal of cofinality at most $\lambda$ and the set $S$ is non-stationary with respect to $\TTT$, then there are no $\lambda$-ascending paths through $\TTT$.
\end{lemma}

In particular, if $\kappa$ is either weakly inaccessible or the successor of a regular cardinal and $\lambda$ is a cardinal with $\lambda^+<\kappa$, then special trees of height $\kappa$ do not contain $\lambda$-ascending paths. In contrast, the first author showed in \cite{ChrisTreesSquareReflection}  that, if $\lambda$ is a singular cardinal, then Jensen's principle $\square_\lambda$ implies the existence of a special tree of height $\lambda^+$ containing a $\cof(\lambda)$-ascent path.


\section{Square principles}

In the following, we recall the definitions of several square principles that will be used in this paper.

\begin{definition}
 Given an uncountable regular cardinal $\kappa$ and a cardinal $1 < \lambda \leq \kappa$, a sequence $\seq{\calC_\alpha}{\alpha < \kappa}$ is a \emph{$\square(\kappa,{<}\lambda)$-sequence} if the following statements hold: 
  \begin{enumerate}
    \item For all $\alpha \in \acc(\kappa)$, $\calC_\alpha$ is a collection of club subsets of $\alpha$ with  $0 < |\calC_\alpha| < \lambda$. 
    
    \item If $\alpha,\beta\in\acc(\kappa)$, $C \in \calC_\beta$, and $\alpha \in\acc(C)$, then $C \cap \alpha \in \calC_\alpha$. 
    
    \item There is no club $D$ in $\kappa$ such that, for all $\alpha \in \acc(D)$, we have $D \cap \alpha \in \calC_\alpha$.
  \end{enumerate}
  
  We let $\square(\kappa,{<}\lambda)$ denote the assertion that there is a $\square(\kappa,{<}\lambda)$-sequence. The principle $\square(\kappa,{<} \lambda^+)$ is typically written as $\square(\kappa, \lambda)$, and  $\square(\kappa,1)$ is written as $\square(\kappa)$. Finally, a sequence $\seq{C_\alpha}{\alpha<\kappa}$ is a \emph{$\square(\kappa)$-sequence} if the sequence $\seq{\{C_\alpha\}}{\alpha<\kappa}$ witnesses that $\square(\kappa)$ holds. 
\end{definition}

We next introduce an indexed version of $\square(\kappa, \lambda)$. The definition is taken from 
\cite{systems} and is a modification of similar indexed square notions studied in \cite{MR1838355} 
and \cite{MR1942302}.

\begin{definition} \label{ind_square_def}
	Let $\lambda < \kappa$ be infinite regular cardinals. A $\square^{\mathrm{ind}}(\kappa, \lambda)$-sequence is a 
  matrix $\vec{\calC} = \seq{C_{\alpha,i}}{\alpha < \kappa, ~ i(\alpha) \leq i < \lambda}$ 
  satisfying the following statements:
  \begin{enumerate}
		\item If $\alpha \in \acc(\kappa)$, then $i(\alpha) < \lambda$. 

		\item If $\alpha \in \acc(\kappa)$ and $i(\alpha) \leq i < \lambda$, then $C_{\alpha,i}$ is a club subset of  $\alpha$. 

		\item If $\alpha \in \acc(\kappa)$ and $i(\alpha) \leq i < j < \lambda$, then $C_{\alpha,i} \subseteq C_{\alpha,j}$. 

		\item If $\alpha,\beta\in\acc(\kappa)$ and $i(\beta) \leq i < \lambda$, then $\alpha \in \acc(C_{\beta,i})$ implies that $i\geq i(\alpha)$ and $ C_{\alpha,i}=C_{\beta,i} \cap \alpha$. 

		\item If $\alpha,\beta\in\acc(\kappa)$ with $\alpha < \beta$, then there is an $i(\beta)\leq i < \lambda$ such that $\alpha \in \acc(C_{\beta,i})$. 

    \item There is no club subset  $D$ of $\kappa$ such that, for all $\alpha \in \acc(D)$, there is $i < \lambda$ such that $C_{\alpha, i}=D \cap \alpha$ holds.
	\end{enumerate}

	We let $\square^{\mathrm{ind}}(\kappa, \lambda)$ denote the assertion that there is a $\square^{\mathrm{ind}}(\kappa, \lambda)$-sequence.
\end{definition}

A proof of the following statement can be found in {\cite[Section 6]{systems}}.

\begin{proposition} \label{threadprop}
  Definition \ref{ind_square_def} is unchanged if we replace condition (6) by the following seemingly weaker condition:
  \begin{enumerate}
    \item[($6'$)]There is no club subset $D$ of $\kappa$ and $i < \lambda$ such that $i\geq i(\alpha)$ and $C_{\alpha,i}=D \cap \alpha$  for all $\alpha \in \acc(D)$. 
  \end{enumerate}
\end{proposition}

It is immediate that the principle $\square^{\mathrm{ind}}(\kappa, \lambda)$ implies $\square(\kappa, \lambda)$. We next show that 
$\square(\kappa)$ implies all relevant instances of $\square^{\mathrm{ind}}(\kappa, \lambda)$.

\begin{theorem} \label{square_building_thm}
  Let $\lambda < \kappa$ be infinite regular cardinals and assume that $\square(\kappa)$ holds. 
  Given a stationary subset $S$ of $\kappa$, there is a $\square^{\mathrm{ind}}(\kappa, \lambda)$-sequence 
  $$\vec{\calC} ~ = ~ \seq{C_{\alpha,i}}{\alpha\in\acc(\kappa), ~ i(\alpha)\leq i<\lambda}$$ 
  with the following properties: 
  \begin{enumerate}
   \item If $i<\lambda$, then the set $\Set{\alpha\in S}{i(\alpha)=i}$ is stationary in $\kappa$. 
   
   \item There is a $\square(\kappa)$-sequence $\seq{D_\alpha}{\alpha<\kappa}$ such that $\acc(D_\alpha)\subseteq\acc(C_{\alpha,i(\alpha)})$ holds for all $\alpha\in\acc(\kappa)$. 
  \end{enumerate} 
\end{theorem}

\begin{proof}
  By a result of Rinot (see {\cite[Lemma 3.2]{MR3271280}}) and our assumptions, there is a $\square(\kappa)$-sequence 
  $\vec{D} = \seq{D_\alpha}{\alpha < \kappa}$ with the property that, for all $\eta < \kappa$, 
  the set $S^\eta = \Set{\alpha\in S}{\min(D_\alpha) = \eta}$ is stationary. 
  Fix a partition $\seq{A_i }{ i < \lambda}$ of $\kappa$ into disjoint, non-empty sets. 
  Given $\alpha \in \acc(\kappa)$, define $i(\alpha)$ to be the unique $i<\lambda$ with $\min(D_\alpha)\in A_i$.  
  Then the set $\Set{\alpha\in S}{i(\alpha)=i}$ is stationary in $\kappa$ for 
  every $i<\kappa$ and, if $\alpha\in\acc(\kappa)$ and $\beta\in\acc(D_\alpha)$, then $i(\alpha)=i(\beta)$
  

 By induction on $\alpha\in\acc(\kappa)$, we define a matrix $\seq{C_{\alpha,i}}{\alpha\in\acc(\kappa), ~ i(\alpha)<i<\lambda}$ 
 satisfying clauses (1)--(5) listed in Definition \ref{ind_square_def} together with the assumption that 
 $\acc(D_\alpha) \subseteq \acc(C_{\alpha, i(\alpha)})$ holds for every $\alpha \in \acc(\kappa)$. 
 In the following, fix a limit ordinal $\alpha<\kappa$ and assume that we already have constructed a matrix with the above properties up to $\alpha$. 
There are a number of cases to consider: 

\smallskip

 \paragraph{\textbf{Case 1:} $\alpha = \omega$.} 
 Define $C_{\omega, i} = \omega$ for all $i(\omega)\leq i <\lambda$.  Then all of the desired requirements are trivially satisfied.

 \smallskip

 \paragraph{\textbf{Case 2a:} $\alpha = \beta + \omega$ for a limit ordinal $\beta$, and $\acc(D_\alpha)=\emptyset$.} 
 Set $j=\max\{i(\alpha),i(\beta)\}$, $C_{\alpha,i}=\Set{\beta+n}{n<\omega}$ for all 
 $i(\alpha)\leq i<j$ and $C_{\alpha,i}=C_{\beta,i}\cup\Set{\beta+n}{n<\omega}$ for all $j\leq i<\lambda$. Then  
 it is easy to see that clauses (1)--(3) and (5) of Definition \ref{ind_square_def} hold. 
 To verify clause (4), suppose $i(\alpha)\leq i<\lambda$ and $\gamma\in\acc(C_{\alpha,i})$. 
 By our construction, it follows that $i\geq j\geq i(\beta)$ and $\gamma\in\acc(C_{\beta,i})\cup\{\beta\}$.
 By the induction hypothesis applied to $\beta$, it follows that $i \geq i(\gamma)$ and
 $C_{\gamma, i} = C_{\beta, i} \cap \gamma = C_{\alpha,i} \cap \gamma$.
 
 \smallskip
 
 \paragraph{\textbf{Case 2b:} $\alpha = \beta + \omega$ for a limit ordinal $\beta$, and $\acc(D_\alpha)\neq\emptyset$.} 
  In the following, we set $\alpha_0=\max(\acc(D_\alpha))\leq\beta$. Then the above remarks show that $i(\alpha_0)=i(\alpha)$. 
  If $\alpha_0<\beta$, then we let $j$ be minimal such that $i(\alpha)\leq j<\lambda$ and $\alpha_0\in\acc(C_{\beta,j})$. Otherwise, we set $j=i(\alpha)$. 
  Then $j\geq i(\beta)$, because either $\alpha_0=\beta$, $\beta\in\acc(D_\alpha)$ and $i(\beta)=i(\alpha_0)=j$ or $\alpha_0<\beta$, $\alpha_0\in\acc(C_{\beta,j})$ and $j\geq i(\beta)$. Define $$C_{\alpha,i} ~ = ~ C_{\alpha_0, i} \cup \{\alpha_0\} \cup \Set{\beta + n}{n < \omega}$$ for all $i(\alpha)\leq i<j$ and $C_{\alpha, i} = C_{\beta, i} \cup \Set{\beta + n }{n < \omega}$ for all $j\leq i<\lambda$. 

Since our induction hypothesis ensures that $C_{\alpha_0,i}\cup\{\alpha_0\}\subseteq C_{\alpha_0,j}\cup\{\alpha_0\}\subseteq C_{\beta,j}$
holds for all $i(\alpha)\leq i<j$, it is easy to see that Clauses (1)--(3) and (5) from Definition \ref{ind_square_def} hold in this case. 
We thus verify clause (4). Suppose $i(\alpha) \leq i < \lambda$ and $\gamma \in \acc(C_{\alpha, i})$. 
If $i(\alpha)\leq i<j$, then either $\gamma=\alpha_0$ or $\gamma\in\acc(C_{\alpha_0,i})$.  
In both instances, we have $i(\gamma)\leq i$ and $C_{\gamma,i}=C_{\alpha_0,i}\cap\gamma=C_{\alpha,i}\cap\gamma$. 
On the other hand, if $j\leq i < \lambda$, then $\gamma\in\acc(C_{\beta,i})\cup\{\beta\}$, so 
$i \geq i(\gamma) $ and $C_{\gamma,i}=C_{\beta,i}\cap\gamma = C_{\alpha, i} \cap \gamma$.  
Finally, the above definitions ensure that $$\acc(D_\alpha) ~ = ~ \acc(D_{\alpha_0})\cup\{\alpha_0\} ~ \subseteq ~ C_{\alpha_0,i(\alpha_0)}\cup\{\alpha_0\} ~  \subseteq ~  C_{\alpha,i(\alpha)}.$$ 

\smallskip

\paragraph{\textbf{Case 3a:} $\alpha$ is a limit of limit ordinals and $\acc(D_\alpha)=\emptyset$.} 
Pick a strictly increasing sequence $\seq{\alpha_n}{n < \omega}$ of limit ordinals cofinal in $\alpha$. 
Given $n < \omega$, let $i(\alpha) \leq j_n < \lambda$ be minimal with $i(\alpha_m) \leq j_n$ 
and $\alpha_m \in \acc(C_{\alpha_{n+1}, j_n})$ for all $m \leq n$. Then our induction hypothesis implies that 
$j_n\leq j_{n+1}$ for all $n<\omega$. Set $j_\omega=\sup_{n<\omega}j_n\leq\lambda$.  
If $i(\alpha)\leq i< j_0$, then we define $C_{\alpha,i}=\Set{\alpha_n}{n<\omega}$. 
Next, if $j_n\leq i< j_{n+1}$ for some $n<\omega$, then $i\geq i(\alpha_n)$, and we define $C_{\alpha,i}=C_{\alpha_n,i}\cup\Set{\alpha_m}{n\leq m<\omega}$. 
Finally, if $j_\omega\leq i<\lambda$, then $i\geq i(\alpha_n)$ for all $n<\omega$, and we define $C_{\alpha, i} = \bigcup\Set{C_{\alpha_n, i}}{n < \omega}$. 

 The above definitions and our induction hypothesis directly ensure that clauses (1) and (5) from Definition \ref{ind_square_def} hold. 
 If $m<n<\omega$ and $i\geq j_{n-1}$, then the above definitions ensure that $\alpha_m\in\acc(C_{\alpha_n,i})\subseteq C_{\alpha_n,i}$ holds. In particular, we have $\Set{\alpha_n}{n<\omega}\subseteq C_{\alpha,i}$ for all $i(\alpha)\leq i<\lambda$. Moreover, we have $C_{\alpha_m,i}=C_{\alpha_n,i}\cap\alpha_m$ for all $j_\omega\leq i<\lambda$ and $m<n<\omega$. This shows that clause (2) from Definition \ref{ind_square_def} holds. 
 Now, fix $i(\alpha)\leq i<j<\lambda$. If either $i$ or $j$ is contained in $[i(\alpha),j_0)\cup[j_\omega,\lambda)$, then the above remark and our induction hypothesis imply that $C_{\alpha,i}\subseteq C_{\alpha,j}$ holds. Next, if there is an $n<\omega$ with $j_n\leq i<j<j_{n+1}$, then our induction hypothesis implies that $C_{\alpha_n,i}\subseteq C_{\alpha_n,j}$ and therefore $C_{\alpha,i}\subseteq C_{\alpha,j}$ also holds in this case. Finally, if there are $m<n<\omega$ with $j_m\leq i<j_{m+1}\leq j_n\leq j<j_{n+1}$, then $\alpha_m\in\acc(C_{\alpha_n},j)$, $C_{\alpha_m,i}\subseteq C_{\alpha_m,j}=C_{\alpha_n,j}\cap\alpha_m$ and, in combination with the above remarks, this implies that $C_{\alpha,i}\subseteq C_{\alpha,j}$ holds. These computations show that clause (3) of Definition \ref{ind_square_def} holds in this case. 
We finally verify clause (4) of Definition \ref{ind_square_def}. To this end, fix $i(\alpha) \leq i < \lambda$ and 
$\gamma \in \acc(C_{\alpha, i})$. Then $i\geq j_0$. If there is an $n<\omega$ with $j_n\leq i<j_{n+1}$, then it follows that $\gamma \in \acc(C_{\alpha_n,i}) \cup \{\alpha_n\}$,
in which case the induction hypothesis implies that $i\geq i(\gamma)$ and $C_{\gamma, i} = C_{\alpha_n,i} \cap \gamma = C_{\alpha,i} \cap \gamma$. In the other case, assume that $j_\omega \leq i < \lambda$ and let $n < \omega$ be least such that $\gamma < \alpha_n$. By the above remarks and our induction hypothesis, we then have $C_{\alpha,i}\cap\gamma=C_{\alpha_n,i}\cap\gamma$, $\gamma\in\acc(C_{\alpha_n,i})$, $i\geq i(\gamma)$ and $C_{\gamma,i}=C_{\alpha_n,i}\cap\gamma=C_{\alpha,i}\cap\gamma$. 

\smallskip

\paragraph{\textbf{Case 3b:} $\alpha$ is a limit of limit ordinals and $\acc(D_\alpha)\neq\emptyset$ is bounded below $\alpha$.} 
  Pick a strictly increasing sequence $\seq{\alpha_n}{ n < \omega}$ of limit ordinals cofinal in $\alpha$ with $\alpha_0 = \max(\acc(D_\alpha))$. Then $\alpha_0\in\acc(D_\alpha)$ implies that $i(\alpha_0)=i(\alpha)$. Define a sequence $\seq{j_n}{n\leq\omega}$ as in Case 3a. 
  If $i(\alpha)\leq i< j_0$, then $i\geq i(\alpha_0)$, and we define $C_{\alpha,i}=C_{\alpha_0,i}\cup\Set{\alpha_n}{n<\omega}$.   
Next, if $j_n\leq i< j_{n+1}$ for some $n<\omega$, then $i\geq i(\alpha_n)$, and we define $C_{\alpha,i}=C_{\alpha_n,i}\cup\Set{\alpha_m}{n\leq m<\omega}$. 
Finally, if $j_\omega\leq i<\lambda$, then $i\geq i(\alpha_n)$ for all $n<\omega$, and we define $C_{\alpha, i} = \bigcup\Set{C_{\alpha_n, i}}{n < \omega}$.

 As above, it is easy to see that clauses (1) and (5) from Definition \ref{ind_square_def} hold. Moreover, we again have $\alpha_m\in\acc(C_{\alpha_n,i})\subseteq C_{\alpha_n,i}$ for all $m<n<\omega$ and $i\geq j_{n-1}$ and, together with our induction hypothesis, this implies that $$C_{\alpha_0,i} ~ \cup ~ \Set{\alpha_n}{n<\omega} ~ \subseteq ~ C_{\alpha,i}$$ for all $i(\alpha)\leq i<\lambda$ and $C_{\alpha_m,i}=C_{\alpha_n,i}\cap\alpha_m$ for all $j_\omega\leq i<\lambda$ and $m<n<\omega$. Hence clause (2) from Definition \ref{ind_square_def} holds. As in Case 3a, clauses (3) and (4) from Definition \ref{ind_square_def} are a direct consequence of our induction hypothesis and the above observations. Finally,  our construction and the induction hypothesis ensure that 
  $$\acc(D_\alpha)  ~ =  ~ \acc(D_{\alpha_0})  \cup  \{\alpha_0\}  ~ \subseteq  ~ \acc(C_{\alpha_0, i(\alpha_0)})   \cup   \{\alpha_0\}  ~ \subseteq  ~ \acc(C_{\alpha, i(\alpha)}).$$ 

\smallskip

\paragraph{\textbf{Case 4:} $\acc(D_\alpha)$ is unbounded in $\alpha$.} 
  Note that, for all $\beta,\gamma\in\acc(D_\alpha)$ with $\beta<\gamma$, our construction and the induction 
  hypothesis imply that $i(\alpha) = i(\beta) = i(\gamma)$, $\gamma \in \acc(C_{\beta, i(\alpha)})$, and therefore $C_{\gamma, i} = C_{\beta, i} \cap \gamma$ for all $i(\alpha) \leq i < \lambda$. If we now define $C_{\alpha, i} = \bigcup\Set{C_{\beta, i}}{\beta \in \acc(D_\alpha)}$ for all $i(\alpha) \leq i < \lambda$, then it is easy to see that clauses (1)--(3) and (5) of Definition \ref{ind_square_def} hold. To verify clause (4), fix 
  $i(\alpha) \leq i < \lambda$ and $\gamma \in \acc(C_{\alpha, i})$. Let $\beta = \min(\acc(D_\alpha) \setminus (\gamma + 1))$. 
  It follows that $\gamma \in \acc(C_{\beta, i})$, so, by the induction hypothesis, we have 
  $i\geq i(\gamma)$ and $C_{\gamma, i} = C_{\beta, i} \cap \gamma = C_{\alpha, i} \cap \gamma$. Finally, our induction hypothesis implies that 
  \begin{equation*}
   \begin{split}
    \acc(D_\alpha) ~ & = ~ \bigcup\Set{\acc(D_\beta)}{\beta\in\acc(D_\alpha)} \\
                                & \subseteq ~ \bigcup\Set{\acc(C_{\beta,i(\beta)})}{\beta\in\acc(D_\alpha)} ~ \subseteq ~ \acc(D_{\alpha,i(\alpha)}).
  \end{split}
 \end{equation*}

 \smallskip

  We have thus constructed a matrix $\vec{C}$ satisfying clauses
  (1)--(5) of Definition \ref{ind_square_def}. We finish the proof by verifying condition (6') from 
  Proposition \ref{threadprop}. To this end, suppose for the sake of a contradiction that there is 
  $i < \lambda$ and a club $E$ in $\kappa$ such that  $i\geq i(\alpha)$ holds  for all $\alpha \in \acc(E)$. Let $T = \Set{\alpha \in S}{i(\alpha) = i+1}$. Since $T$ is stationary in $\kappa$, there is $\alpha \in \acc(E) \cap T$.   But then $i\geq i(\alpha)=i+1$, a contradiction.  
\end{proof}

In the proof of the following result, we use Todor\v{c}evi\'{c}'s method of \emph{walks on ordinals} to construct trees with ascent paths from suitable square principles.

\begin{theorem} \label{indexed_square_thm}
  If $\lambda<\kappa$ are infinite regular cardinals and $\square^{\mathrm{ind}}(\kappa,\lambda)$ holds, then there is a $\kappa$-Aronszajn tree with a $\lambda$-ascent path. 
\end{theorem}

\begin{proof}
 Fix a $\square^{\mathrm{ind}}(\kappa,\lambda)$-sequence $\vec{C}=\seq{C_{\alpha,i}}{\alpha<\kappa, ~ i(\alpha)\leq i<\lambda}$. 
 Given $i<\lambda$, let $\vec{C}_i=\seq{C^i_\alpha}{\alpha<\kappa}$ denote the unique $C$-sequence (see {\cite[Section 1]{MR908147}}) with $C^i_\alpha = C_{\alpha, i(\alpha)}$ for all $\alpha \in \acc(\kappa)$ with $i < i(\alpha)$ and $C^i_\alpha = C_{\alpha, i}$ for all $\alpha \in \acc(\kappa)$ with $i(\alpha) \leq i$.  
 For each $i < \lambda$, recursively define $\map{\rho_0^{\vec{C}_i}}{[\kappa]^2}{{^{{<}\omega} \kappa}}$ as in \cite{MR908147} 
 by letting, for all $\alpha < \beta < \kappa$, 
 \[
  \rho_0^{\vec{C}_i}(\alpha, \beta) = \langle \otp(C^i_\beta \cap \alpha) \rangle ^\frown \rho_0^{\vec{C}_i}(\alpha, \min(C^i_\beta 
  \setminus \alpha)),
 \]
 subject to the boundary condition $\rho_0^{\vec{C}_i}(\alpha, \alpha) = \emptyset$ for all $\alpha < \kappa$.
 Given $i<\lambda$, we set  $$\TTT_i ~ = ~ \TTT(\rho_0^{\vec{C}_i})  ~ = ~  (\Set{\rho_0^{\vec{C}_i}(~ \cdot ~, \beta) \restriction \alpha}{\alpha \leq \beta < \kappa}, \subset).$$

 \begin{claim}\label{branch_claim}
  If $i<\lambda$, then the tree $\TTT_i$ has no cofinal branches. 
 \end{claim}

 \begin{proof}[Proof of the Claim]
   Suppose $\TTT_i$ has a cofinal branch. By \cite[(1.7)]{MR908147}, there is a club $D$ in $\kappa$ and a $\xi < \kappa$ with the property that, for all $\alpha < \kappa$, there is $\beta(\alpha) \geq \alpha$  with  $$D \cap \alpha ~ = ~  C^i_{\beta(\alpha)} \cap [\xi, \alpha).$$   
   
Given $\alpha \in \acc(D)$, we have $\beta(\alpha)\in\acc(\kappa)$ and there is a $j(\alpha) < \lambda$ with the property that $C^i_{\beta(\alpha)} = C_{\beta(\alpha), j(\alpha)}$. Since $\vec{C}$ is a $\square^{\mathrm{ind}}(\kappa, \lambda)$-sequence and $\alpha\in\acc(C^i_{\beta(\alpha)})=\acc(C_{\beta(\alpha),j(\alpha)})$ for all $\alpha\in\acc(D)$, we can conclude that $$D \cap \alpha ~ = ~  C^i_{\beta(\alpha)} \cap [\xi, \alpha) ~ = ~ C_{\beta(\alpha),j(\alpha)} \cap [\xi, \alpha) ~ = ~ C_{\alpha,j(\alpha)}\cap [\xi, \alpha).$$  
Fix an unbounded subset $E$ of $\acc(D)$ and $j < \lambda$ with $j(\alpha)=j$ for all $\alpha \in E$. 
Given $\alpha,\beta\in E$ with $\xi<\alpha<\beta$, we have $\alpha \in \acc(C_{\beta, j})$ and therefore   $C_{\alpha, j} = C_{\beta, j} \cap \alpha$. If we define $C^* = \bigcup\Set{C_{\alpha, j}}{\xi<\alpha \in E}$, then $C^*$ is a club in $\kappa$ and, for all $\alpha \in \acc(C^*)$, we have $C_{\alpha, j} = C^* \cap \alpha$, contradicting the fact that $\vec{C}$ is a $\square^{\mathrm{ind}}(\kappa, \lambda)$-sequence.
 \end{proof}

 \begin{claim} \label{level_claim}
  If $i<\lambda$ and $\alpha<\kappa$, then the $\alpha$-th level of $\TTT_i$ has cardinality less than $\kappa$. 
 \end{claim}

 \begin{proof}[Proof of the Claim]
   Let $\mathcal{D}^i_\alpha = \Set{C^i_\beta \cap \alpha}{\alpha \leq \beta < \kappa}$.
   By \cite[(1.3)]{MR908147}, the $\alpha$-th level of $\TTT_i$ has cardinality at most 
   $\betrag{\mathcal{D}^i_\alpha} + \aleph_0$. If $D \in \mathcal{D}^i_\alpha$, then $D$ is the union 
   of a finite subset of $\alpha$ and a set of the form $C_{\gamma, j}$, where $\gamma \leq \alpha$ 
   and $j < \lambda$. There are only $\betrag{\alpha}$-many finite subsets of $\alpha$ and only $\max\{\lambda, \betrag{\alpha}\}$-many sets of the form $C_{\gamma, j}$, where $\gamma \leq \alpha$ and $j < \lambda$. In combination, this shows that $\betrag{\mathcal{D}^i_\alpha} \leq \max\{\lambda, \betrag{\alpha}\} < \kappa$.
 \end{proof}

 Now, define $\TTT$ to be the unique tree with the following properties: 
 \begin{enumerate}
  \item The underlying set of $\TTT$ is the collection of all pairs $\langle i,t\rangle$ such that $i<\lambda$, $t\in\TTT_i$ and $\height{t}{\TTT_i}$ is a limit ordinal. 
  
  \item Given nodes $\langle i,t\rangle$ and $\langle j,u\rangle$ in $\TTT$, we have $\langle i,t\rangle\leq_\TTT\langle j,u\rangle$ if and only if $i=j$ and $t\leq_{\TTT_i}u$.
 \end{enumerate}
 
 Then Claims \ref{branch_claim} and \ref{level_claim} directly imply that $\TTT$ is a $\kappa$-Aronszajn tree. 
 Given $\alpha<\kappa$, we define $$\Map{c_\alpha}{\lambda}{\TTT}{i}{\langle i, ~ \rho_0^{\vec{C}_i}( ~ \cdot ~ , ~ \omega\cdot\alpha)\rangle}.$$

 \begin{claim}
  The sequence $\seq{c_\alpha}{\alpha<\kappa}$ is a $\lambda$-ascent path in $\TTT$. 
 \end{claim}

 \begin{proof}[Proof of the Claim]
  Note that, for all $i < \lambda$ and all $\alpha < \beta < \kappa$, if $C^i_{\omega \cdot \alpha} = 
  C^i_{\omega \cdot \beta} \cap (\omega \cdot \alpha)$, then $$\rho^{\vec{C}_i}_0( ~ \cdot ~ , \omega\cdot\alpha) ~ 
  = ~ \rho^{\vec{C}_i}_0( ~ \cdot ~ , \omega\cdot\beta) \restriction \omega\cdot\alpha.$$ Let $j_{\alpha, \beta} < \lambda$ be least such that 
  $\omega \cdot \alpha \in \acc(C_{\omega \cdot \beta, j_{\alpha, \beta}})$. Then, for all $j_{\alpha, \beta} \leq i < \lambda$, 
  we have $C^i_{\omega \cdot \alpha} = C_{\omega \cdot \alpha, i}$, $C^i_{\omega\cdot\beta} = C_{\omega\cdot\beta, i}$, 
  and $C_{\omega\cdot\alpha, i} = C_{\omega\cdot\beta, i} \cap \omega\cdot\alpha$. It follows that, for all $\alpha < \beta < \kappa$ 
  and all $j_{\alpha, \beta} \leq i < \lambda$, we have $c_\alpha(i) <_\TTT c_\beta(i)$.
 \end{proof}

 This completes the proof of the theorem. 
\end{proof}

The statement of Theorem \ref{square_ascent_path_thm} now follows directly from Theorems \ref{square_building_thm} and \ref{indexed_square_thm}.


\section{Forcing preliminaries}\label{section:Forcing_preliminaries}

In this section, we review some forcing posets designed to add and thread square sequences. We also recall constructions 
to make weak compactness or the tree property indestructible under mild forcing.


\subsection{Forcing square sequences}

We first recall the notion of strategic closure.

\begin{definition}
  Let $\PPP$ be a partial order (with maximal element $\mathbbm{1}_\PPP$) and let $\beta$ be an ordinal. 
  \begin{enumerate}
    \item $\Game_\beta(\PPP)$ is the two-player game of perfect information in which Players I and II alternate playing conditions from       $\PPP$ to attempt to construct a $\leq_\PPP$-decreasing sequence $\seq{p_\alpha}{\alpha < \beta}$. 
      Player I plays at all odd stages, and Player II plays at all even (including limit) stages. Player II is 
      required to play $p_0 = \mathbbm{1}_\PPP$. If, during the course of play, a limit ordinal $\alpha < \beta$ is 
      reached such that $\seq{p_\xi}{\xi < \alpha}$ has no lower bound in $\PPP$, then Player I wins. Otherwise, 
      Player II wins. 
      
    \item $\PPP$ is \emph{$\beta$-strategically closed} if Player II has a winning strategy in $\Game_\beta(\PPP)$.
  \end{enumerate}
\end{definition}

\begin{definition} \label{square_forcing_def}
  Given cardinals $1 < \lambda \leq \kappa$ with $\kappa$ regular and uncountable, we let $\SSS(\kappa,{<} \lambda)$ denote the partial order defined by the following clauses: 
  \begin{enumerate}
   \item A condition in $\SSS(\kappa,{<} \lambda)$ is a sequence $p = \seq{\calC^p_\alpha}{\alpha \leq \gamma^p}$ with $\gamma^p \in \acc(\kappa)$ such that the following statements hold for all $\alpha,\beta \in \acc(\gamma^p + 1)$: 
   \begin{enumerate}
    \item $\calC^p_\alpha$ is a collection of club subsets of $\alpha$ with $0 < \betrag{\calC^p_\alpha} < \lambda$. 
    
    \item If $C \in \calC^p_\beta$ with $\alpha \in \acc(C)$, then $C \cap \alpha \in \calC^p_\alpha$.
  \end{enumerate}
   
   \item The ordering of $\SSS(\kappa,{<} \lambda)$ is given by end-extension, i.e., $q \leq_{\SSS(\kappa,{<} \lambda)} p$ holds  if and only if $\gamma^q \geq \gamma^p$ and,   for all $\alpha \leq \gamma^p$, $\calC^q_\alpha = \calC^p_\alpha$.
  \end{enumerate}  
\end{definition}

In the following, we will usually write $\SSS(\kappa, \lambda)$ instead of    $\SSS(\kappa, {<} \lambda^+)$.  
The proof of the following lemma is standard and follows, for example, from the proofs of \cite[Proposition 33 and Lemma 35]{MR3129734}.

\begin{lemma}
  Let $1 < \lambda \leq \kappa$ be cardinals with $\kappa$ regular and uncountable. 
  \begin{enumerate}
    \item $\SSS(\kappa,{<} \lambda)$ is $\omega_1$-closed. 
    
    \item $\SSS(\kappa,{<} \lambda)$ is $\kappa$-strategically closed. 
    
    \item If $G$ is $\SSS(\kappa,{<} \lambda)$-generic over $\VV$, then $\bigcup G$ is a $\square(\kappa, {<} \lambda)$-sequence in $\VV[G]$.
  \end{enumerate}
\end{lemma}

We next consider a forcing poset meant to add a thread to a $\square(\kappa, {<} \lambda)$-sequence.

\begin{definition}\label{definition:POthreadSquare}
 Let $1 < \lambda \leq \kappa$ be cardinals with $\kappa$ regular and uncountable, and let $\vec{\calC} 
  = \seq{\calC_\alpha}{\alpha < \kappa}$ be a $\square(\kappa, {<} \lambda)$-sequence. We let  $\TTT(\vec{\calC})$ denote the partial order whose underlying set is $\bigcup\Set{\calC_\alpha}{\alpha \in \acc(\kappa)}$ and whose ordering is given by end-extension, i.e. $D \leq_{\TTT(\vec{\calC})} C$ holds if and only if $C=D \cap \sup(C)$.
\end{definition}

In what follows, if $\PPP$ is a partial order and $\theta$ is a cardinal, then we let $\PPP^\theta$ denote the full-support product of $\theta$ copies of $\PPP$.

\begin{lemma} \label{dense_closed_lemma_1}
 Let $1 < \lambda \leq \kappa$ be cardinals with $\kappa$ regular and uncountable, let $\dot{\calC}$ be the canonical $\SSS(\kappa,{<}\lambda)$-name for the  $\square(\kappa, {<} \lambda)$-sequence added by $\SSS(\kappa,{<}\lambda)$ and let $\dot{\TTT}$ be the canonical $\SSS(\kappa,{<}\lambda)$-name   for $\TTT(\dot{\calC})$. For all $0 < \theta < \lambda$, the partial order $\SSS(\kappa,{<}\lambda)* \dot{\TTT}^\theta$ has  a dense $\kappa$-directed closed subset.
\end{lemma}

\begin{proof}
  Fix $\theta < \lambda$, and let $\UUU$ denote the set of all $\langle p,\dot{f}\rangle \in \SSS(\kappa,{<}\lambda) * \dot{\TTT}^\theta$  such that, for all $\xi < \theta$, there is a $D_\xi \in \calC^p_{\gamma^p}$ with  
  $p \Vdash_{\SSS(\kappa,{<}\lambda)}\anf{\dot{f}(\xi) = \check{D}_\xi}.$ Then standard arguments show that $\UUU$ is dense in   $\SSS(\kappa,{<}\lambda) * \dot{\TTT}^\theta$ and $\kappa$-directed closed.
\end{proof}

There is also a natural forcing notion to add a $\square^{\mathrm{ind}}(\kappa, \lambda)$-sequence by initial segments.

\begin{definition} \label{ind_square_forcing_def}
  Given infinite, regular cardinals $\lambda < \kappa$, we let $\PPP(\kappa, \lambda)$  denote the partial order defined by the following clauses: 
  \begin{enumerate}
   \item A condition in $\PPP(\kappa, \lambda)$ is a matrix 
   $$p ~ = ~  \seq{C^p_{\alpha, i}}{\alpha \leq \gamma^p, ~ i(\alpha)^p \leq i < \lambda}$$ 
  with  $\gamma^p \in \acc(\kappa)$ and $i(\alpha)^p < \lambda$ for all $\alpha \leq \gamma^p$ such that the following statements hold for all $\alpha,\beta \in \acc(\gamma^p+1)$: 
   \begin{enumerate}
		\item If $i(\alpha)^p \leq i < \lambda$, then $C^p_{\alpha, i}$ is a club subset of  $\alpha$. 
		
		\item If  $i(\alpha)^p \leq i < j < \lambda$, then $C^p_{\alpha, i} \subseteq C^p_{\alpha, j}$. 
		
		\item If $\alpha \in \acc(C^p_{\beta, i})$ for some $i(\beta)^p \leq i < \lambda$, then we have $i\geq i(\alpha)^p$ and $ C^p_{\alpha, i} = C^p_{\beta, i} \cap \alpha$. 
		
		\item If $\alpha < \beta$, then there is $i < \lambda$ with $\alpha \in \acc(C^p_{\beta, i})$.  
  \end{enumerate}
   
   \item The ordering of $\PPP(\kappa, \lambda)$ is given by end-extension, i.e., $q \leq_{\PPP(\kappa, \lambda)} p$ holds  if and only if $\gamma^q \geq \gamma^p$ and the following statements hold for all $\alpha \in \acc(\gamma^p+1)$: 
   	\begin{enumerate}
		\item $i(\alpha)^q = i(\alpha)^p$. 
		
		\item If $i(\alpha)^p \leq i < \lambda$, then $C^q_{\alpha, i} = C^p_{\alpha, i}$. 
	\end{enumerate}
  \end{enumerate}  
\end{definition}

The following results are proven in {\cite[Section 7]{systems}} or follow directly from the proofs presented there.

\begin{lemma}\label{lemma:PropertiesGenericIndexedSquare}
  Let $\lambda < \kappa$ be infinite, regular cardinals. 
  \begin{enumerate}
    \item $\PPP(\kappa, \lambda)$ is $\lambda$-directed closed. 
    
    \item $\PPP(\kappa, \lambda)$ is $\kappa$-strategically closed. 
    
    \item If $G$ is $\PPP(\kappa, \lambda)$-generic over $\VV$, then $\bigcup G$ is a $\square^{\mathrm{ind}}(\kappa, \lambda)$-sequence in $\VV[G]$ and for every regular cardinal $\mu<\kappa$ and every $i<\lambda$, the set $$\Set{\alpha\in E^\kappa_\mu}{\exists p\in G ~ [\alpha\leq\gamma^p\wedge i(\alpha)^p=i]}$$ is stationary in $\kappa$ in $\VV[G]$.
  \end{enumerate}
\end{lemma}

We finally introduce a forcing notion to add a thread through a $\square^{\mathrm{ind}}(\kappa, \lambda)$-sequence.

\begin{definition}\label{definition:POthreadIndexedSquare}
  Let $\vec{\calC} = \seq{C_{\alpha, i}}{\alpha < \kappa, ~ i(\alpha) \leq i < \lambda}$ be a $\square^{\mathrm{ind}}(\kappa,\lambda)$-sequence. 
Given $i < \lambda$,  let  $\TTT_i(\vec{\calC}\hspace{1.6pt})$ denote the partial order whose underlying set is $$\Set{C_{\alpha, i}}{\alpha \in \acc(\kappa), ~ i(\alpha)\leq i}$$ and whose ordering is given by end-extension.  
\end{definition}

The following result is proven in {\cite[Section 3]{hayut_lh}}.

\begin{lemma}\label{lem: amalgamation for indexed square}
  Let $\lambda < \kappa$ be infinite, regular cardinals, let $\dot{C}$ 
  be the canonical $\PPP(\kappa, \lambda)$-name for the generic $\square^{\mathrm{ind}}(\lambda, \kappa)$-sequence and, for $i < \kappa$, let $\dot{\TTT}_i$   be a $\PPP(\kappa,\lambda)$-name for $\TTT_i(\dot{C})$.  
\begin{enumerate}
 \item For all $i < \lambda$, $\PPP(\kappa,\lambda) * \dot{\TTT}_i$ has a dense $\kappa$-directed closed subset. 

\item Let $G$ be $\PPP(\kappa,\lambda)$-generic over $\VV$ and let $\bigcup G=\seq{C_{\alpha,i}}{\alpha<\kappa, ~ i(\alpha)\leq i<\lambda}$ be the generic $\square^{\mathrm{ind}}(\lambda, \kappa)$-sequence.  Given $i < j < \lambda$, the map $$\Map{\pi_{i,j}}{\dot{\TTT}_i^G}{\dot{\TTT}_j^G}{C_{\alpha, i}}{C_{\alpha, j}}$$ is a forcing projection in $\VV[G]$.  
\end{enumerate} 
\end{lemma}

\begin{lemma} \label{deciding_lemma}
 Let $\lambda < \kappa$ be infinite, regular cardinals, let $G$ be $\PPP(\kappa,\lambda)$-generic over $\VV$, and let $\vec{\calC}$ be the generic $\square^{\mathrm{ind}}(\lambda, \kappa)$-sequence in $\VV[G]$. 
 Assume that, in $\VV[G]$, for each $i < \lambda$, 
  $\dot{x}_i$ is a $\TTT_i(\vec{\calC}\hspace{1.6pt})$-name for an element of $\VV[G]$. 
 Given    $t \in \TTT_0(\vec{\calC}\hspace{1.6pt})$, there is a condition $s \leq_{\TTT_0(\vec{\calC}\hspace{1.6pt})} t$ such that $\pi_{0,i}(s)$ decides the value of $\dot{x}_i$ for all $i < \lambda$. 
\end{lemma}

\begin{proof}
  Work in $\VV[G]$. For all $i < \lambda$, let $\DDD_i$ be the set of $s \in \TTT_0(\vec{\calC}\hspace{1.6pt})$ such that $\pi_{0,i}(s)$ 
  decides the value of $\dot{x}_i$. 
  Since $\pi_{0,i}$ is a projection, the set $\DDD_i$ is dense and open in $\TTT_0(\vec{\calC}\hspace{1.6pt})$. 
  Since $\PPP(\kappa,\lambda) * \dot{\TTT}_0$ has a dense $\kappa$-directed closed subset in $\VV$, it follows  that the partial order $\TTT_0(\vec{\calC}\hspace{1.6pt})$ is $\kappa$-distributive in $\VV[G]$. We can thus find $s \in \bigcap\Set{\DDD_i}{i < \lambda}$   with $s \leq_{\TTT_0(\vec{\calC}\hspace{1.6pt})} t$. Then the condition $s$ is as desired.
\end{proof}


\subsection{Indestructibility}

In this subsection, we outline techniques for arranging so that weak compactness and the tree property 
necessarily hold at a cardinal $\kappa$ after forcing with the partial order $\Add{\kappa}{ 1}$ that adds a Cohen subset to $\kappa$. These methods are 
well-known, so we will just sketch the constructions or refer the reader elsewhere for details.

A classical result of Silver shows that weak compactness of a weakly compact cardinal can be made indestructible under forcing with $\Add{\kappa}{1}$. The proof of this result is described in {\cite[Section 3]{MR495118}}. A modern presentation of these arguments can be found in {\cite[Example 16.2]{MR2768691}}.

\begin{lemma}[Silver]
  Let $\kappa$ be a weakly compact cardinal. Then there is a forcing extension $\VV[G]$ such that $\kappa$ is weakly compact in $\VV[G,H]$ whenever $H$ is $\Add{\kappa}{1}$-generic over $\VV[G]$. 
\end{lemma}

Note that the conclusion of the above result implies that $\kappa$ is also weakly compact in $\VV[G]$, because forcing with $\kappa$-closed partial orders does not add branches to $\kappa$-Aronszajn trees.


\begin{definition}
  Given an uncountable, regular cardinal $\kappa$, we let $\mathrm{TP}^+(\kappa)$ denote the conjunction of $\mathrm{TP}(\kappa)$ and the statement that  for every regular cardinal $\lambda$ with $\lambda^+ < \kappa$, if  $\vec{b}$ is a $\lambda$-ascent path through a tree $\TTT$ of height $\kappa$, then there is a true cofinal branch through $\vec{b}$. 
\end{definition}

  It is easy to see that, if $\kappa$ is a weakly compact cardinal, then $\mathrm{TP}^+(\kappa)$ holds. The proof of the following result is a slight modification of the presentation of Mitchell's consistency proof of the tree property at $\aleph_2$ in {\cite[Section 23]{MR2768691}}.

\begin{lemma}
  Let $\kappa$ be a weakly compact cardinal and let $\mu < \kappa$ be an infinite, regular cardinals with $ \mu=\mu^{{<}\mu}$. Then there is a forcing extension $\VV[G]$ of the ground model $\VV$ such that the following statements hold:
  \begin{enumerate}
    \item $\VV$ and $\VV[G]$ have the same cofinalities below $(\mu^+)^\VV$. 
    
    \item $(\mu^+)^\VV=(\mu^+)^{\VV[G]}$, $\kappa  = (2^\mu)^{\VV[G]} =   (\mu^{++})^{\VV[G]}$ and $\kappa=(\kappa^{<\kappa})^{\VV[G]}$. 
    
    \item If $H$ is $\Add{\kappa}{1}$-generic over $\VV[G]$, then $\mathrm{TP}^+(\kappa)$ holds in $\VV[G,H]$. 
  \end{enumerate}
\end{lemma}

\begin{proof}[Sketch of the proof]
 Given $\alpha\leq\kappa$, let $\dot{\RRR}_\alpha$ denote the canonical $\Add{\mu}{\alpha}$-name with the property that, whenever $G$ is $\Add{\mu}{\alpha}$-generic over $\VV$, then $\dot{\RRR}_\alpha^G=\Add{\mu^+}{1}^{\VV[G]}$. 
By induction on $\alpha\leq\kappa$, we define a sequence $\seq{\QQQ(\alpha)}{\alpha \leq \kappa}$ of partial orders with the property that, if $\alpha\leq\kappa$ is inaccessible and $G$ is $\QQQ(\alpha)$-generic over $\VV$, then $\VV$ and $\VV[G]$ have the same cofinalities below $(\mu^+)^\VV$,  $(\mu^+)^\VV=(\mu^+)^{\VV[G]}$, $\alpha  = (2^\mu)^{\VV[G]} =   (\mu^{++})^{\VV[G]}$, and $\alpha=(\alpha^{<\alpha})^{\VV[G]}$. 

  Fix $\beta \leq \kappa$ and assume that the partial order $\QQQ(\alpha)$ with the above properties has been defined for all $\alpha < \beta$. Given $\alpha<\beta$ inaccessible, let $\dot{\SSS}_\alpha$ denote the canonical $\QQQ(\alpha)$-name with the property that, whenever $G$ is $\QQQ(\alpha)$-generic over $\VV$, then $\dot{\SSS}_\alpha^G=\Add{\alpha}{1}^{\VV[G]}$.   We define the underlying set of the partial order $\QQQ(\beta)$ to consist of triples $\langle p, f, g\rangle$ that satisfy the following statements:  
  \begin{enumerate}
    \item[(a)] $p$ is a condition in $\Add{\mu}{\beta}$. 
    
    \item[(b)] $f$ is a partial function on $\beta$ of cardinality at most $\mu$ with the property that for all $\alpha\in\dom(f)$, $\alpha$ is a successor ordinal and $f(\alpha)$ is an $\Add{\mu}{\alpha}$-name for a condition in $\dot{\RRR}_\alpha$. 
    
    \item[(c)] $g$ is a partial function on $\beta$ of cardinality at most $\mu$ with the property that for all $\alpha\in\dom(g)$, $\alpha$ is an inaccessible cardinal and $g(\alpha)$ is a $\QQQ(\alpha)$-name for a condition in $\dot{\SSS}_\alpha$.  
  \end{enumerate}

  Given triples $\langle p_0, f_0, g_0\rangle$ and  $\langle p_1, f_1, g_1\rangle$ satisfying the above statements, we define $\langle p_1, f_1, g_1\rangle \leq_{\QQQ(\beta)}   \langle p_0, f_0, g_0\rangle$ to hold if and only if the following statements hold:
  \begin{enumerate}
    \item[(i)] $p_1 \leq_{\Add{\mu}{\beta}} p_0$. 
    
    \item[(ii)] $\dom(f_1) \supseteq \dom(f_0)$ and $\dom(g_1) \supseteq \dom(g_0)$. 
    
    \item[(iii)] If $\alpha \in \dom(f_0)$, then $$p_1 \restriction \alpha \Vdash_{\Add{\mu}{\alpha}} \anf{f_1(\alpha)   \leq_{\dot{\RRR}_\alpha} f_0(\alpha)}.$$ 
      
    \item[(iv)] If $\alpha \in \dom(g_0)$, then $$\langle p_1 \restriction \alpha, ~  f_1 \restriction \alpha, ~ g_1 \restriction \alpha\rangle     \Vdash_{\QQQ(\alpha)} \anf{g_1(\alpha) \leq_{\dot{\SSS}_\alpha}g_0(\alpha)}.$$
  \end{enumerate}
  By representing $\QQQ(\beta)$ as a projection of a product $\Add{\mu}{\beta}\times\PPP$ for some $\mu^+$-closed partial order $\PPP$, it is possible to show that $\QQQ(\beta)$ satisfies the assumptions listed above. 
  
  Set $\QQQ = \QQQ(\kappa)$. Given $\alpha < \kappa$ inaccessible, the canonical map 
  $$\Map{\pi_\alpha}{\QQQ}{\QQQ(\alpha)*\dot{\SSS}_\alpha}{\langle p,f,g\rangle}{\langle\langle p 
  \restriction \alpha, ~  f \restriction \alpha, ~ g \restriction \alpha\rangle, ~ g(\alpha)\rangle}$$ 
  is a projection. Let $G$ be $\PPP$-generic over $\VV$, let $H$ by $\Add{\kappa}{1}$-generic over $V[G]$, 
  and let $G_\alpha*H_\alpha$ be the filter on  $\PPP(\alpha)*\dot{\SSS}_\alpha$ induced by $G$ via $\pi_\alpha$.
  Suppose for sake of contradiction that $\TTT$ is a $\kappa$-Aronszajn tree (with $\kappa$ as an underlying set) in $\VV[G,H]$. 
  A routine application of weak compactness yields an inaccessible cardinal $\alpha < \kappa$ 
  such that $\TTT \restriction \alpha \in \VV[G_\alpha, H_\alpha]$ and $\TTT \restriction \alpha$ is an $\alpha$-Aronszajn tree 
  in $\VV[G_\alpha, H_\alpha]$. However, standard arguments show that $\mu^+$-approximation holds between $\VV[G_\alpha,H_\alpha]$ and $\VV[G]$ 
  (see {\cite[Definition 21.2]{MR2768691}}).  It follows that $\alpha$-Aronszajn trees in $\VV[G_\alpha, H_\alpha]$ 
  cannot gain cofinal branches in $\VV[G, H]$. However, $\TTT \restriction \alpha$ does have a cofinal branch in $\VV[G, H]$, namely 
  the set of predecessors of any element of $\TTT(\alpha)$, which is a contradiction.
  
  Similarly, if $\lambda\leq\mu$ and $\vec{b}$ is a $\lambda$-ascent path through a $\kappa$-tree $\TTT$ in $\VV[G,H]$  
  with no true cofinal branch, then there is an inaccessible $\alpha < \kappa$ such that $\TTT \restriction \alpha, ~  
  \vec{b} \restriction \alpha \in \VV[G_\alpha, H_\alpha]$ and there is no true cofinal branch through $\vec{b} \restriction \alpha$ in 
  $\VV[G_\alpha, H_\alpha]$. As above, an appeal to $\mu^+$-approximation yields a contradiction.  
\end{proof}


\section{Consistency results for trees}\label{ConsResultsTrees}

Building on the results of the previous section, we prove consistency results that will provide upper bounds for the consistency strength of two interactions between ascent paths and special trees listed in Table \ref{table:Implications_Special_Ascent_Succ_Reg}.

\begin{theorem} \label{no_special_no_ascent_thm}
  Let  $\kappa$ be an uncountable cardinal with the property that $\kappa=\kappa^{{<}\kappa}$ and $\mathbbm{1}_{\Add{\kappa}{1}}\Vdash\mathrm{TP}^+(\check{\kappa})$. Then the following statements hold in a cofinality-preserving forcing extension of the ground model:
  \begin{enumerate}
    \item There are $\kappa$-Aronszajn trees. 
    
    \item There are no special $\kappa$-trees. 
    
    \item For all $\lambda$ with $\lambda^+ < \kappa$, there are no $\kappa$-Aronszajn trees   with $\lambda$-ascent paths.
  \end{enumerate}
\end{theorem}

\begin{proof}
  Let $\PPP = \SSS(\kappa, 2)$ be the forcing from Definition \ref{square_forcing_def}   that adds a $\square(\kappa, 2)$-sequence. Let $\dot{\calC}$ be the canonical  $\PPP$-name for the generically-added $\square(\kappa, 2)$-sequence, and let $\dot{\QQQ}$ be the canonical $\PPP$-name for the partial order $\TTT(\dot{\calC})$ defined in Definition \ref{definition:POthreadSquare}. By Lemma \ref{dense_closed_lemma_1}, 
  if $\theta \in \{1,2\}$, then the partial order $\PPP * \dot{\QQQ}^\theta$ has a dense $\kappa$-directed closed subset. 
  Moreover, since $\kappa=\kappa^{{<}\kappa}$ holds, the dense subset constructed in the proof of Lemma \ref{dense_closed_lemma_1} 
  has size $\kappa$ and is hence forcing equivalent to $\Add{\kappa}{1}$.

  Let $G$ be $\PPP$-generic over $\VV$, and set $\QQQ=\dot{\QQQ}^G$. 
  We claim that the above statements hold in $\VV[G]$. We first note that, since $\square(\kappa, 2)$ holds in $\VV[G]$, there are $\kappa$-Aronszajn trees in this model. To verify clause (2), note that our assumptions and the above computations imply that, if $H$ is $\QQQ$-generic over $\VV[G]$, then $\kappa$ is a regular cardinal in $\VV[G,H]$ and every $\kappa$-Aronszajn tree in $\VV[G]$ has a cofinal branch in $\VV[G,H]$. This observation directly implies that there are no special 
  $\kappa$-trees in $\VV[G]$. 
  
  To verify clause (3), suppose for sake of a contradiction that, in $\VV[G]$, there is a $\kappa$-Aronszajn   tree $\TTT$ ,  a cardinal $\lambda$ with $\lambda^+ < \kappa$, and a $\lambda$-ascent path $\vec{b} = \seq{b_\alpha}{\alpha < \kappa}$ through $\TTT$. 
  Given $\alpha < \beta < \kappa$, let $i_{\alpha,\beta} < \lambda$ be least such that, for all $i_{\alpha, \beta} \leq j < \lambda$, we have $b_\alpha(j) <_\TTT b_\beta(j)$.  Since $\QQQ$ forces $\mathrm{TP}^+(\kappa)$ to hold, there are $\QQQ$-names $\dot{I}$ and $\dot{B}$ such that 
  $$\mathbbm{1}_\QQQ\Vdash\anf{\textit{The pair $\langle\dot{I}, \dot{B}\rangle$ is a true cofinal branch through $\vec{b}$}}$$ holds in $\VV[G]$.  
  Let $H_0 \times H_1$ be $\QQQ^2$-generic over $\VV[G]$, and, given $\varepsilon < 2$, set $I_\varepsilon=\dot{I}^{H_\varepsilon}$ and $B_\varepsilon=\dot{B}^{H_\varepsilon}$.   Note that, since $\PPP * \dot{\QQQ}^2$ has a dense $\kappa$-directed closed subset in $\VV$, 
  $\kappa$ remains a regular cardinal in $\VV[G,H_0,H_1]$. 

  Work now in $\VV[G,H_0,H_1]$. For each $\alpha < \kappa$ and $\varepsilon < 2$, set $\alpha_\epsilon = \min(B_\varepsilon\setminus (\alpha + 1))$. Using the regularity of  $\kappa$ and the fact that $\lambda < \kappa$, we find $\max(I_0,I_1)\leq I_* < \lambda$ and unbounded subsets   $B_0^* \subseteq B_0$ and $B_1^* \subseteq B_1$ with $i_{\alpha, \alpha_{1-\varepsilon}} \leq I_*$ for all $\varepsilon<2$ and all $\alpha \in B_\varepsilon^*$.    Given $\varepsilon < 2$, set $A_\varepsilon = \Set{\alpha < \kappa}{i_{\alpha, \alpha_\epsilon} \leq I_*}$.
  
  \begin{claim}
    For $\varepsilon < 2$, the pair $\langle I_*, A_\varepsilon\rangle$ is a true cofinal branch through $\vec{b}$.
  \end{claim}
  
  \begin{proof}[Proof of the Claim]
    Fix $\alpha,\beta\in A_\varepsilon$ with $\alpha < \beta$. By definition of $A_\varepsilon$ and the fact     that $I_* \geq I_\epsilon$, it follows that $b_\alpha(i) <_\TTT b_{\alpha_\varepsilon}(i) \leq_\TTT b_{\beta_\varepsilon}(i)$ and $b_\beta(i) <_\TTT b_{\beta_\varepsilon}(i)$ for all $I_* \leq i < \lambda$.    Since $\leq_\TTT$ is a tree order, it follows that $b_\alpha(i) <_\TTT b_\beta(i)$  for all $I_* \leq i < \lambda$ and therefore $\langle I_*, A_\varepsilon\rangle$ satisfies Clause (\ref{c3c}) of Definition    \ref{ascent_path_def}.

    Next, fix $\beta \in A_\varepsilon$ and $\alpha < \beta$ such that $b_\alpha(i) <_\TTT b_\beta(i)$     for all $I_* \leq j < \lambda$. As above, it follows that $b_\alpha(i) <_\TTT b_\beta(i) <_\TTT b_{\beta_\varepsilon}(i)$ and $b_{\alpha_\varepsilon}(i) \leq_\TTT b_{\beta_\varepsilon}(i)$  for all $I_* \leq i < \lambda$.     Thus, again by the fact that $\leq_\TTT$ is a tree order, this implies that $b_\alpha(i) <_\TTT b_{\alpha_\varepsilon}(i)$  for all $I_* \leq i < \lambda$ and hence $\alpha \in A_\varepsilon$.  This allows us to conclude that     the pair $\langle I_*, A_\varepsilon\rangle$ also satisfies Clause (\ref{c3d}) of Definition \ref{ascent_path_def}.
  \end{proof}

  \begin{claim}
    $A_0 = A_1$. 
  \end{claim}

  \begin{proof}[Proof of the Claim]
    We show $A_0 \subseteq A_1$. The proof of the reverse inclusion is symmetric. Thus, fix $\alpha \in A_0$ and set $\beta = \min(B_0^* \setminus \alpha_1)$. Given $I_* \leq i < \lambda$, we then have $$b_\alpha(i) ~  <_\TTT  ~ b_{\alpha_0}(i) ~  \leq_\TTT ~ b_\beta(i) ~  <_\TTT ~ b_{\beta_1}(i),$$    where the last relation holds because $\beta \in B_0^*$.     But we also have $b_{\alpha_1}(i) <_\TTT b_{\beta_1}(i)$ for all $I_* \leq i < \lambda$, because $\alpha_1,\beta_1\in B_1$ and $I_*\geq I_1$.  But then, again using the fact that $\leq_\TTT$ is a tree order, we can conclude that 
    $b_\alpha(i) <_\TTT b_{\alpha_1}(i)$ holds for all $I_* \leq i < \lambda$. This shows that $\alpha$ is an element of $A_1$.
  \end{proof}

 Given $\varepsilon < 2$, the set $A_0=A_\varepsilon$ is definable from $I_*$ and $B_\varepsilon$. Hence the pair $\langle I_*,A_0\rangle$ is a member of $\VV[G][H_\varepsilon]$ for all $\varepsilon<2$. By the Product Lemma, it follows that $\langle I_*,A_0\rangle$ is contained in $\VV[G]$. But then $\Set{b_\alpha(I_*)}{\alpha \in A_0}$ 
  is a cofinal branch through $\TTT$, contradicting the assumption that $\TTT$ is a $\kappa$-Aronszajn tree in   $\VV[G]$.
\end{proof}

Next, we show that, consistently, there are $\kappa$-Aronszajn trees and all such trees contain ascent paths of small width.

\begin{proof}[Proof of Theorem \ref{all_ascent_path_thm}]
  Let $\lambda < \kappa$ be infinite, regular cardinals such that $\kappa=\kappa^{{<}\kappa}$ and $\mathbbm{1}_{\Add{\kappa}{1}}\Vdash\mathrm{TP}(\check{\kappa})$.   Define $\PPP = \PPP(\kappa, \lambda)$ to be the forcing notion from Definition \ref{ind_square_forcing_def} that adds a $\square^{\mathrm{ind}}(\kappa, \lambda)$-sequence. Let $\dot{\calC}$ be the   canonical $\PPP$-name for the generically-added $\square^{\mathrm{ind}}(\kappa, \lambda)$-sequence,   and, for all $i < \lambda$, let $\dot{\QQQ}_i$ be a $\PPP$-name for the partial order $\TTT_i(\dot{\calC})$ defined in Definition \ref{definition:POthreadIndexedSquare}.  
  By Lemma \ref{lem: amalgamation for indexed square}, if $i < \lambda$, then the partial order $\PPP * \dot{\QQQ}_i$ 
  has a dense, $\kappa$-directed closed subset and the assumption that $\kappa=\kappa^{{<}\kappa}$ implies that this dense subset is forcing equivalent to $\Add{\kappa}{1}$. Therefore, by our assumptions, 
we know that $\mathbbm{1}_{\PPP * \dot{\QQQ}_i}\Vdash\mathrm{TP}(\check{\kappa})$ for all  $i < \lambda$. 

  Let $G$ be $\PPP$-generic over $V$ and, for all $i < \lambda$, set $\QQQ_i=\dot{\QQQ}_i^G$. Given $i < j < \lambda$, let $\map{\pi_{i,j}}{\QQQ_i}{\QQQ_j}$ be the projection given by Lemma \ref{lem: amalgamation for indexed square}. 
  Let $$\dot{\calC}^G ~ = ~ \seq{C_{\alpha, i}}{\alpha < \kappa, ~ i(\alpha) \leq i < \lambda}$$ be the 
  realization of $\dot{\calC}$.
  We claim that $\VV[G]$ is the desired forcing extension. 
  Since $\square(\kappa, \lambda)$ holds in $\VV[G]$, there are $\kappa$-Aronszajn trees in this model. We thus 
  verify requirement (2). To this end, work in $\VV[G]$ and fix a $\kappa$-Aronszajn tree $\TTT $.   For all $i < \lambda$, we have $\mathbbm{1}_{\QQQ_i}\Vdash\mathrm{TP}(\check{\kappa})$ and hence we can fix a $\QQQ_i$-name $\dot{B}_i$ for a cofinal branch in $\TTT$.    We may assume that  $\dot{B}_i$ is forced to be $\leq_\TTT$-downward closed, i.e., that $\dot{B}_i$ is forced  to meet every level of $\TTT$.   For each $\alpha < \kappa$, use Lemma \ref{deciding_lemma} to find  $q_\alpha \in \QQQ_0$ such that, for all $i < \lambda$, there is a node $x_{\alpha, i} \in \TTT(\alpha)$ with the property that $$\pi_{0,i}(q_\alpha)\Vdash_{\QQQ_i}\anf{\check{x}_{\alpha,i}\in\dot{B}_i 
  \cap \check{\TTT}(\check{\alpha})}.$$ 
 Define a sequence of functions $\vec{b} = \seq{\map{b_\alpha}{\lambda}{\TTT(\alpha)}}{\alpha < \kappa}$   by setting $b_\alpha(i) = x_{\alpha, i}$ for all $\alpha < \kappa$ and $i < \lambda$. 

  We claim that $\vec{b}$ is a $\lambda$-ascent path through $\TTT$. To see this, fix   $\alpha < \beta < \kappa$. Pick $\gamma_\alpha, \gamma_\beta \in \acc(\kappa)$ such that 
  $q_\alpha = C_{\gamma_\alpha, 0}$ and $q_\beta = C_{\gamma_\beta, 0}$. Without loss of generality, we may assume that $\gamma_\alpha < \gamma_\beta$; the other cases are treated similarly.
  Fix an $i< \lambda$ such that $\gamma_\alpha \in \acc(C_{\gamma_\beta, i})$  and hence $\pi_{0,j}(q_\beta) \leq_{\QQQ_j} \pi_{0,j}(q_\alpha)$ for all $i\leq j<\lambda$. This shows that $\pi_{0,j}(q_\beta) \Vdash_{\QQQ_j}\anf{x_{\alpha, j}, x_{\beta, j} \in \dot{B}_j}$ holds for all $i\leq j<\lambda$.   Since $\dot{B}_i$ is a name for a branch through $\TTT$, this implies that  $$b_\alpha(j) ~ = ~ x_{\alpha, j} ~ <_\TTT ~ x_{\beta, j} ~ = ~  b_\beta(j)$$ holds for all $i \leq j < \lambda$.    Therefore, $\vec{b}$ is a $\lambda$-ascent path through $\TTT$. 
\end{proof}

With the help of a result form \cite{ascending_paths}, it is easy to see that the conclusion of Theorem \ref{all_ascent_path_thm} implies that $\kappa$ is a weakly compact cardinal in $\LL$.

\begin{lemma}\label{lemma:AllAscentSquareFails}
 Let $\kappa$ be a regular cardinal and let $\lambda$ be an infinite cardinal with $\lambda^+<\kappa$. If every $\kappa$-Aronszajn tree contains a $\lambda$-ascending path, then $\square(\kappa)$ fails. 
\end{lemma}

\begin{proof}
 Assume, towards a contradiction, that there is a $\square(\kappa)$-sequence $\vec{C}$ and let $\TTT=\TTT(\rho_0^{\vec{C}})$ denote the tree of full codes of walks through $\vec{C}$ defined in {\cite[Section 1]{MR908147}} (as in the proof of Theorem \ref{indexed_square_thm}). Then the results of \cite{MR908147} show that $\TTT$ is a $\kappa$-Aronszajn tree and {\cite[Lemma 4.5]{ascending_paths}} implies that $\TTT$ does not contain a $\lambda$-ascending path, contradicting our assumption.  
\end{proof}

It has long been known that, for regular cardinals $\kappa > \aleph_1$, the principle $\square(\kappa)$ does not imply the existence of a special $\kappa$-Aronszajn tree. For example, this is the case in $\LL$ if $\kappa$ is a Mahlo cardinal that is not weakly compact, and it will remain true in the forcing extension 
of $\LL$ by $\Col{\aleph_1}{{<}\kappa}$, in which $\kappa = \aleph_2$. We now show that $\square(\kappa)$ does not even imply the existence of a $\kappa$-Aronszajn tree $\TTT$ such that there is a stationary subset of $\kappa$ that is non-stationary with respect to $\TTT$. In particular, this shows that the various trees that will be constructed from the principle $\square(\kappa)$ in the proofs of Theorem \ref{square_productivity_thm} and \ref{KnasterLayered_thm} in Section \ref{section:ChainConditions} cannot be assumed to be $\kappa$-Aronszajn trees.

\begin{theorem}
  Let $\kappa > \aleph_1$ be a regular cardinal with the property that $\kappa=\kappa^{<\kappa}$  and $\mathbbm{1}_{\Add{\kappa}{1}}\Vdash\mathrm{TP}(\check{\kappa})$. Then the following statements hold in  a cofinality-preserving forcing extension of the ground model:
  \begin{enumerate}
    \item $\square(\kappa)$ holds. 
    
    \item If $\TTT$ is a $\kappa$-Aronszajn tree and $S$ is a stationary subset of $\kappa$, then  $S$ is stationary with respect to $\TTT$.
  \end{enumerate}
\end{theorem}

\begin{proof}
  Let $\SSS = \SSS(\kappa, 1)$ be  the forcing from Definition \ref{square_forcing_def}   that adds a $\square(\kappa)$-sequence. 
  Let $\dot{C}$ be the canonical $\SSS$-name for the generically-added square sequence, and let $\dot{\RRR}$ be an $\SSS$-name for the threading forcing $\TTT(\dot{C})$ defined in Definition \ref{definition:POthreadSquare}.    By Lemma \ref{dense_closed_lemma_1} and our assumptions, the partial order $\SSS * \dot{\RRR}$ has a dense $\kappa$-directed closed subset.

 Let $G$ be $\SSS$-generic over $\VV$ and set $\RRR=\dot{\RRR}^G$. By \cite[Lemma 3.4 and Corollary 3.5]{hayut_lh}, there is a forcing iteration 
  $\seq{\PPP_\eta, \dot{\QQQ}_\xi}{\eta \leq \kappa^+, \xi < \kappa^+}$ in $\VV[G]$  with supports of  size less than $\kappa$, such that, letting $\PPP = \PPP_{\kappa^+}$, the following statements hold:  
  \begin{enumerate}
    \item[(a)] If $\eta \leq \kappa^+$ and $\dot{\PPP}_\eta$ is the canonical $\SSS$-name for $\PPP_\eta$ in $\VV$, then $\SSS *(\dot{\PPP}_\eta \times \dot{\RRR})$       has a $\kappa$-directed closed dense subset in $\VV$. Moreover, if $\eta < \kappa^+$, then this      subset can be assumed to have size $\kappa$. 
      
    \item[(b)] $\PPP$ satisfies the $\kappa^+$-chain condition in $\VV[G]$. 
          
    \item[(b)] If $H$ is $\PPP$-generic over $\VV[G]$, then $\square(\kappa)$ holds in $\VV[G,H]$ and, for every stationary subset $E$ of $\kappa$ in $\VV[G,H]$, there is a condition $r$ in $\RRR$ such that $r\Vdash_\RRR\anf{\textit{$\check{E}$  is stationary in $\check{\kappa}$}}$ holds in $\VV[G,H]$. 
  \end{enumerate}

  Let $H$ be $\PPP$-generic over $\VV[G]$. For $\eta < \kappa^+$, let 
  $H_\eta$ be the $\PPP_\eta$-generic filter induced by $H$. We claim that $\VV[G,H]$ is the desired model.   Thus, work in $\VV[G,H]$ and suppose, for the sake of contradiction, that $\TTT$ is a $\kappa$-Aronszajn and $E$ is a stationary subset of $\kappa$ that is non-stationary with respect to $\TTT$. By the properties of $\PPP$, 
  we can find $r \in \RRR$ with $r \Vdash_\RRR\anf{\textit{$\check{E}$ is a stationary subset of $\kappa$}}$.  

  Since the tree $\TTT$, the subset $E$, and the maps witnessing that $E$ is non-stationary with respect to $\TTT$ can all be coded by subsets of $\VV$ of cardinality $\kappa$  in $\VV[G,H]$, the fact that $\PPP$ satisfies the $\kappa^+$-chain condition in $\VV[G]$   implies that there is an $\eta < \kappa^+$ such that $E,\TTT \in\VV[G,H_\eta]$ and $E$ is non-stationary with respect to $\TTT$ in $\VV[G,H_\eta]$.  
  
  Let $K$ be $\RRR$-generic   over $\VV[G,H]$ with $r \in K$.   Since the partial order $\SSS * (\dot{\PPP}_\eta \times \dot{\RRR})$ has a dense   $\kappa$-directed closed subset of size $\kappa$ in $\VV$, our assumptions imply that the tree property holds at   $\kappa$ in $V[G,H_\eta,K]$. However, $E$ remains stationary in $V[G,H,K]$ and thus, \emph{a fortiori},  in $\VV[G,H_\eta,K]$. Moreover, the maps witnessing that $E$ is non-stationary with respect to $\TTT$ obviously   persist in $V[G,H_\eta,K]$, so $E$ remains non-stationary with respect to $\TTT$ in $V[G,H_\eta,K]$, and so,   by Fact \ref{non_stationary_fact}, $\TTT$ is a $\kappa$-Aronszajn tree in $V[G,H_\eta,K]$, contradicting   the fact that the tree property holds at $\kappa$.
\end{proof}


\section{Provable implications}\label{section:ProvableImplications}

In this section, we piece things together to provide a complete explanation of Table \ref{table:Implications_Special_Ascent_Succ_Reg} from the end of the 
Introduction, thus completing the picture of the interaction between special trees and trees with ascent paths 
at successors of regular cardinals. Throughout this section, we will work under the assumption that there are 
$\aleph_2$-Aronszajn trees.


\subsection{Inconsistencies}

We first note that, by Lemma \ref{lucke_lemma}, an $\aleph_2$-Aronszajn tree with an $\aleph_0$-ascent 
path cannot be special. This immediately implies that $\forall \TTT.S(\TTT)$ is incompatible with 
$\exists \TTT.A(\TTT)$ and that $\forall \TTT.A(\TTT)$ is incompatible with $\exists \TTT.S(\TTT)$, 
so the three boxes in the upper left of the table are inconsistent.


\subsection{Lower bounds}

We now deal with lower bounds for the boxes in the bottom row and right column of the table. First, by
Theorem \ref{square_ascent_path_thm}, $\square(\aleph_2)$ implies the existence of an $\aleph_2$-Aronszajn 
tree with an $\aleph_0$-ascent path. By results of Jensen and Todor\v{c}evi\'{c} (see {\cite[Section 6]{jensen_fine_structure}} and {\cite[(1.10)]{MR908147}}), a failure of $\square(\aleph_2)$ implies that $\aleph_2$ is weakly compact in $L$. Therefore, the consistency of $\forall \TTT.\neg A(\TTT)$ implies the consistency of a weakly compact cardinal, which takes care of 
all three boxes in the right column of the table.

By Lemma \ref{lemma:AllAscentSquareFails}, the assumption that $\forall \TTT.A(\TTT)$ holds implies a failure of $\square(\kappa)$ and, as above, we can conclude that the consistency of $\forall \TTT.A(\TTT)$ implies the consistency of a weakly compact cardinal.

Finally, Jensen's principle $\square_{\aleph_1}$ implies the existence of a special $\aleph_2$-Aronszajn tree and, by a result of Jensen in \cite{jensen_fine_structure}, the failure of $\square_{\aleph_1}$ implies that $\aleph_2$ is Mahlo in $L$. In particular, the consistency of $\forall \TTT.\neg S(\TTT)$ implies the consistency of a Mahlo cardinal, thus finishing our derivation of lower bounds.


\subsection{Upper bounds}

We finally deal with upper bounds. First, Theorem \ref{all_ascent_path_thm} shows that the 
consistency of $\forall \TTT.A(\TTT)$ follows from the consistency 
of a weakly compact cardinal. Since $\forall \TTT.A(\TTT)$ implies $\forall \TTT. \neg S(\TTT)$, 
it follows that the consistency of the conjunction of these two statements follows from the consistency 
of a weakly compact cardinal.

Next, in \cite{MR603771}, Laver and Shelah prove that the consistency of a weakly compact cardinal 
implies the consistency of Souslin's Hypothesis at $\aleph_2$. A straightforward and well-known strengthening of 
their argument yields that the consistency of a weakly compact cardinal in fact implies the consistency of 
$\forall \TTT.S(\TTT)$. Since $\forall \TTT.S(\TTT)$ implies $\forall \TTT. \neg A(\TTT)$, it follows 
that the consistency of the conjunction of these two statements follows from the consistency of a 
weakly compact cardinal. Moreover, it immediately follows that the conjunction of $\exists \TTT.S(\TTT)$ and 
$\forall \TTT. \neg A(\TTT)$ follows from the consistency of a weakly compact cardinal.

Since $\square_{\aleph_1}$ implies the existence both of a special $\aleph_2$-Aronszajn tree and an $\aleph_2$-Aronszajn 
tree with an $\aleph_0$-ascent path, it follows that the consistency of the conjunction of $\exists \TTT.S(\TTT)$ and 
$\exists \TTT.A(\TTT)$ does not require large cardinals.

Suppose $\kappa$ is the least Mahlo cardinal in $\LL$, and force over $\LL$ with the partial order constructed by Mitchell in \cite{MR0313057} for $\kappa$. Then there are no special $\aleph_2$-Aronszajn trees in the extension. Since $\kappa$ is not weakly compact in $\LL$, $\square(\aleph_2)$ holds in the extension, so, by Corollary \ref{square_ascent_path_cor}, there is an $\aleph_2$-Aronszajn 
tree with an $\aleph_0$-ascent path. It follows that the consistency of the conjunction of $\forall \TTT.\neg S(\TTT)$ and $\exists \TTT.A(\TTT)$ follows from the consistency of a Mahlo cardinal.

Finally, Theorem \ref{no_special_no_ascent_thm} shows that the consistency of the conjunction of $\forall \TTT.\neg S(\TTT)$ 
and $\forall \TTT.\neg A(\TTT)$ follows from the consistency of a weakly compact cardinal, thus completing the table.


\section{Chain conditions}\label{section:ChainConditions}

The first part of this section is devoted to the proofs of Theorem \ref{square_productivity_thm} and \ref{KnasterLayered_thm}. In the second part, we will use ideas from Section \ref{ConsResultsTrees} to construct a model of set theory in which the class of all partial orders satisfying the $\kappa$-chain condition exhibits an interesting product behavior.

\begin{definition}
  Suppose that $\vec{\calC} = \seq{C_{\alpha,i}}{\alpha < \kappa,   ~ i(\alpha) \leq i < \lambda}$ is a $\square^{\mathrm{ind}}(\kappa, \lambda)$-sequence, and $i < \lambda$.  
  \begin{enumerate}
    \item We define $S^{\vec{\calC}}_i = \Set{\alpha \in\acc(\kappa)}{i(\alpha) = i}$ and $S^{\vec{\calC}}_{{\leq}i}  = \Set{\alpha \in\acc(\kappa)}{i(\alpha) \leq i}$. The sets $S^{\vec{\calC}}_{{<} i}$, $S^{\vec{\calC}}_{{>} i}$, etc. are defined analogously. 
    
    \item We let $\TTT_i^{\vec{\calC}}$ denote the tree with underlying set $S^{\vec{\calC}}_{{\leq} i}$ and $$\alpha <_{\TTT_i^{\vec{\calC}}} \beta ~ \Longleftrightarrow ~ \alpha \in \acc(C_{\beta, i})$$ for all  $\alpha, \beta \in S^{\vec{\calC}}_{{\leq} i}$. 
  \end{enumerate}
\end{definition}

\begin{lemma}\label{lemma:TreeIndexedSquareSpecialSubset}
  Let $\vec{\calC}=\langle C_{\alpha, i} \mid \alpha < \kappa, ~ i(\alpha) \leq i < \lambda \rangle$ 
  be a $\square^{\mathrm{ind}}(\kappa,\lambda)$-sequence and let $i<\lambda$. If the tree $\TTT_i^{\vec{\calC}}$ has height $\kappa$, then the set $S^{\vec{\calC}}_{{>}i}$ is non-stationary with respect to $\TTT_i^{\vec{\calC}}$.  
\end{lemma}

\begin{proof}
  Set $\TTT=\TTT_i^{\vec{\calC}}$ and $S = S^{\vec{\calC}}_{>i}$. 
  Let $A$ denote the set of all $\alpha\in\TTT\restriction S$ with $\acc(C_{\alpha,i})\cap\height{\alpha}{\TTT}\neq\emptyset$. 
  Given $\alpha\in A$, define $s(\alpha)=\sup(\acc(C_{\alpha,i})\cap\height{\alpha}{\TTT})$. 
  Since $\alpha\geq\height{\alpha}{\TTT}\in S$, we have $\alpha\notin S$ and therefore $\alpha>\height{\alpha}{\TTT}$ holds for all $\alpha\in A$. This shows that $s(\alpha)\leq\height{\alpha}{\TTT}<\alpha$, $s(\alpha)\in\acc(C_{\alpha,i})$, $i(s(\alpha))\leq i$ and $s(\alpha)\in\TTT$. Moreover, we have $\acc(C_{\alpha,i})\cap S=\emptyset$ and  hence $s(\alpha)<\height{\alpha}{\TTT}$ for all $\alpha\in A$.

 Given $\beta\in\TTT$, define $A_\beta$ to be the set of all $<_\TTT$-minimal elements $\gamma$ in $s^{{-}1}``\{\beta\}$. Let $B$ denote the set of all $\alpha\in A$ with $\alpha\notin A_{s(\alpha)}$, and let $\map{r}{\TTT\restriction S}{\TTT}$ denote the unique function with the following properties: 
 \begin{enumerate}
  \item If $\alpha\in B$, then $r(\alpha)$ is the unique element $\gamma$ of $A_{s(\alpha)}$ with $\gamma<_\TTT \alpha$. 
  
  \item If $\alpha\in A\setminus B$, then $r(\alpha)=s(\alpha)$. 
  
  \item If $\alpha\in(\TTT\restriction S)\setminus A$ is not minimal in $\TTT$, then we define $r(\alpha)=\min(\acc(C_{\alpha,i}))$. 
  
  \item If $\alpha\in\TTT\restriction S$ is minimal in $\TTT$, then we define $r(\alpha)=\alpha$. 
\end{enumerate}

 By the above remarks, the function $r$ is regressive on $\TTT\restriction S$. Fix $\gamma\in\TTT$ and let $\map{c_\gamma}{r^{{-}1}``\{\gamma\}}{\omega\times\kappa}$ denote the unique function with the following properties: 
 \begin{enumerate}
  \item[(a)] If $\alpha\in B$, then $c_\gamma(\alpha)=\langle 0,\height{\alpha}{\TTT}\rangle$. 
  
  \item[(b)] If $\alpha\in A\setminus B$, then $c_\gamma(\alpha)=\langle 1,0\rangle$. 
  
  \item[(c)] If $\alpha\in(\TTT\restriction S)\setminus A$ is not minimal in $\TTT$, then $c_\gamma(\alpha)=\langle 2,\height{\alpha}{\TTT}\rangle$. 
  
  \item[(d)] If $\alpha\in\TTT\restriction S$ is minimal in $\TTT$, then $c_\gamma(\alpha)=\langle 3,0\rangle$. 
 \end{enumerate}
 
 Then $c_\gamma$ is injective on $<_\TTT$-chains in $r^{{-}1}``\{\gamma\}$. 
 If $\alpha\in\dom(c_\gamma)\setminus A$ is not minimal in $\TTT$, then we have  $\acc(C_{\alpha,i})\neq\emptyset$,  $\acc(C_{\alpha,i})\cap\height{\alpha}{\TTT}=\emptyset$ and therefore $$\height{\alpha}{\TTT} ~ \leq ~  \min(\acc(C_{\alpha,i})) ~ = ~   r(\alpha) ~  = ~  \gamma.$$
  Next, pick $\alpha\in B\cap\dom(c_\gamma)$. Then we have $\height{\alpha}{\TTT}<\alpha$, $\gamma\in A_{s(\alpha)}\subseteq A\subseteq\TTT\restriction S$ and $\gamma<_\TTT \alpha$. This implies that $\height{\gamma}{\TTT}<\min\{\height{\alpha}{\TTT},\gamma\}$,  $C_{\gamma,i}=C_{\alpha,i}\cap\gamma$ and $$\max(\acc(C_{\alpha,i})\cap\height{\alpha}{\TTT}) ~ = ~ \max(\acc(C_{\gamma,i})\cap\height{\gamma}{\TTT}) ~ = ~ \max(\acc(C_{\alpha,i})\cap\height{\gamma}{\TTT}).$$ In particular, we have $\acc(C_{\alpha,i})\cap[\height{\gamma}{\TTT},\height{\alpha}{\TTT})=\emptyset$ and therefore $\height{\alpha}{\TTT}\leq\gamma$, because otherwise we would have $\gamma\in\acc(C_{\alpha,i})\cap[\height{\gamma}{\TTT},\height{\alpha}{\TTT})$. These computations show that the range of $c_\gamma$ has cardinality strictly less than $\kappa$. 
\end{proof}

The following type of partial order will be crucial in our construction of Knaster partial orders with interesting product behavior.

\begin{definition}\label{definition:SpecializiationForcing}
  Given a tree $\TTT$, we let $\PPP(\TTT)$ denote the partial order consisting of finite partial functions $\pmap{f}{\TTT}{\omega}{part}$ that are injective on $<_{\TTT}$-chains and that are ordered by reverse inclusion. 
\end{definition}

Remember that a tree $\TTT$ is \emph{extensional at limit levels} if $\pred_\TTT(s)\neq\pred_\TTT(t)$ holds for every limit ordinal $\alpha$ and all $s,t\in\TTT(\alpha)$ with $s\neq t$.

\begin{lemma} \label{knaster_lemma}
 Let $\kappa$ be an uncountable regular cardinal, let $\mu<\kappa$ be a (possibly finite) cardinal with $\nu^\mu<\kappa$ for all $\nu<\mu$, let $S$ be a subset of $E^\kappa_{{>}\mu}$ that is stationary in $\kappa$, and let $\seq{\TTT^\gamma}{\gamma<\mu}$ be a sequence of trees of height at most $\kappa$ that are extensional at limit levels. Assume that $S$ is non-stationary with respect to $\TTT^\gamma$ for every $\gamma<\lambda$ with $\height{\TTT^\gamma}{}=\kappa$. Then the full support product $\prod_{\gamma<\mu}\PPP(\TTT^\gamma)$ is $\kappa$-Knaster.  
\end{lemma}

\begin{proof}
 Since two conditions in a partial order of the form $\PPP(\TTT)$ are compatible if and only if their union is a condition, it suffices to prove the statement for trees of cardinality $\kappa$, because we can always consider trees of this form that are given by the downward closures of the unions of the domains of $\kappa$-sequences of conditions. Note that every such tree is isomorphic to a tree $\TTT$ with the property that the underlying set of $\TTT$ is a subset of $\kappa\times\kappa$, we have $\TTT(\alpha)\subseteq\{\alpha\}\times\kappa$ for all $\alpha<\kappa$ and $\langle\alpha_0,\beta_0\rangle <_\TTT\langle\alpha_1,\beta_1\rangle$ implies $\beta_0<\beta_1$ for all nodes $\langle\alpha_0,\beta_0\rangle$ and $\langle\alpha_1,\beta_1\rangle$ in  $\TTT$. 

 Assume that all trees in the above sequence are of this form and fix a sequence $\seq{p_\alpha}{\alpha<\kappa}$ of conditions in $\prod_{\gamma<\mu}\PPP(\TTT^\gamma)$. If $\gamma<\mu$ and $\height{\TTT^\gamma}{}=\kappa$, then we also fix functions $\map{r_\gamma}{\TTT^\gamma\restriction S}{\TTT^\gamma}$ and $\seq{\map{c^\gamma_t}{r_\gamma^{{-}1}``\{t\}}{\kappa^\gamma_t}}{t\in\TTT^\gamma}$ witnessing the non-stationarity of $S$ with respect to $\TTT^\gamma$. In the other case, if $\gamma<\mu$ and $\height{\TTT^\gamma}{}<\kappa$, then we let $\map{r_\gamma}{\TTT^\gamma\restriction S}{\TTT^\gamma}$ denote the unique regressive function with $\ran{r_\gamma}\subseteq\TTT^\gamma(0)$ and we set $c^\gamma_{r_\gamma(t)}(t)=\height{t}{\TTT^\gamma}$ for all $t\in\TTT^\gamma\restriction S$. 
   Now, define $\DDD_\gamma$ to be the set of all conditions $p$ in $\PPP(\TTT^\gamma)$ with the property that for all $t,u\in\dom(p)$ with $\height{t}{\TTT^\gamma}<\height{u}{\TTT^\gamma}$, there is $s\in\dom(p)$ with $\height{s}{\TTT^\gamma}=\height{t}{\TTT^\gamma}$ and $s<_{\TTT^\gamma} u$. Then it is easy to see that $\DDD_\gamma$ is a dense subset of $\PPP(\TTT^\gamma)$ and we can pick a sequence $\seq{q_\alpha}{\alpha<\kappa}$ of conditions in $\prod_{\gamma<\mu}\PPP(\TTT^\gamma)$ with $q_\alpha(\gamma)\leq_{\PPP(\TTT^\gamma)}p_\alpha(\gamma)$ and $q_\alpha(\gamma)\in\DDD_\gamma$ for all $\alpha<\kappa$ and $\gamma<\mu$. 
  By our assumptions on the trees $\TTT^\gamma$, there is a club $C$ of limit ordinals in $\kappa$ with the property that $\dom(q_\alpha(\gamma))\subseteq\beta\times\beta$ holds for all $\alpha,\beta\in C$ with $\alpha<\beta$ and all $\gamma<\mu$. 
    For all $\alpha\in C$ and $\gamma<\mu$, fix an injective enumeration $\seq{t^{\alpha,\gamma}_k}{k<n_{\alpha,\gamma}}$ of the finite set $\Set{t\in\TTT^\gamma(\alpha)}{\exists u\in\dom(q_\alpha(\gamma)) ~ t\leq_{\TTT^\gamma} u}$. Since the trees $\TTT^\gamma$ are extensional at limit levels and $S$ is a subset of $E^\kappa_{{>}\mu}$, there is a regressive function $\map{\rho}{C\cap S}{\kappa}$ and a matrix $$\seq{\map{\iota_{\alpha,\gamma}}{n_{\alpha,\gamma}}{\TTT^\gamma(\rho(\alpha))}}{\alpha\in C\cap S, ~ \gamma<\mu}$$ of injections with $\dom(q_\alpha(\gamma))\cap\TTT^\gamma_{{<}\alpha}\subseteq\TTT^\gamma_{{<}\rho(\alpha)}$, $\dom(q_\alpha(\gamma))\cap(\alpha\times\alpha)\subseteq \rho(\alpha)\times\rho(\alpha)$ and $r_\gamma(t^{\alpha,\gamma}_k)<_{\TTT^\gamma} \iota_{\alpha,\gamma}(k) <_{\TTT^\gamma} t^{\alpha,\gamma}_k$ for all $\alpha\in C\cap S$, $\gamma<\mu$ and $k<n_\alpha$. 
    
 In this situation, the assumption that $\nu^\mu<\kappa$ holds for all $\nu<\kappa$ yields   
 a stationary subset $E$ of $C\cap S$, 
 an ordinal $\xi<\kappa$, 
 a sequence $\seq{n_\gamma}{\gamma<\mu}$ of natural numbers, 
 a subset $K\subseteq\mu\times\omega$, 
 a subset $H\subseteq \mu\times\xi$ and 
 a sequence $\seq{D_\gamma}{\gamma<\mu}$ of finite subsets of $\kappa\times\kappa$
 such that the following statements hold  for all $\alpha,\beta\in E$ and $\langle\gamma,k\rangle\in K$: 
 \begin{enumerate}
  \item $n_\gamma=n_{\alpha,\gamma}$ and $\rho(\alpha)=\xi$. 
  
  \item $K=\Set{\langle \gamma,k\rangle}{\gamma<\mu, ~ k<n_\gamma, ~ \iota_{\alpha,\gamma}(k)\in\alpha\times\alpha}$. 
  
  \item $H=\Set{\langle\gamma,\height{t}{\TTT^\gamma}\rangle}{\gamma<\mu, ~ t\in\dom(q_\alpha(\gamma))\cap\TTT^\gamma_{{<}\alpha}}$. 
  
  \item $D_\gamma=\dom(q_\alpha(\gamma))\cap(\alpha\times\alpha)$ and $q_\alpha(\gamma)\restriction D_\gamma= q_\beta(\gamma)\restriction D_\gamma$. 
  
  \item $\iota_{\alpha,\gamma}(k)=\iota_{\beta,\gamma}(k)$, $r_\gamma(t^{\alpha,\gamma}_k)=r_\gamma(t^{\beta,\gamma}_k)$ and $c^\gamma_{r_\gamma(t^{\alpha,\gamma}_k)}(t^{\alpha,\gamma}_k)=c^\gamma_{r_\gamma(t^{\beta,\gamma}_k)}(t^{\beta,\gamma}_k)$. 
 \end{enumerate} 

 Now, pick $\alpha,\beta\in E$ with $\alpha<\beta$ and assume for a contradiction that the conditions $q_\alpha$ and $q_\beta$ are incompatible in $\prod_{\gamma<\mu}\PPP(\TTT^\gamma)$. 
 Then there is a $\gamma<\mu$ such that $q_\alpha(\gamma)\cup q_\beta(\gamma)$ is not a condition in $\PPP(\TTT^\gamma)$ and hence there are $<_{\TTT^\gamma}$-comparable nodes $t,u\in\TTT^\gamma$ such that $t\in\dom(q_\alpha(\gamma))\setminus\dom(q_\beta(\gamma))$, $u\in\dom(q_\beta(\gamma))\setminus\dom(q_\alpha(\gamma))$ and $q_\alpha(\gamma)(t)=q_\beta(\gamma)(u)$. 
 But then $t<_{\TTT^\gamma}u$, because otherwise $u<_{\TTT^\gamma} t\in\beta\times\beta$ would imply that $u\in\dom(q_\beta(\gamma))\cap(\beta\times\beta)=D_\gamma\subseteq\dom(q_\alpha(\gamma))$. 
 Next, assume that $\height{t}{\TTT^\gamma}<\alpha$. 
 Then $\langle\gamma,\height{t}{\TTT^\gamma}\rangle\in H$ and there is a $t_0\in\dom(q_\beta(\gamma))$ with $\height{t_0}{\TTT^\gamma}=\height{t}{\TTT^\gamma}<\height{u}{\TTT^\gamma}$.  Since $q_\beta(\gamma)\in\DDD_\gamma$, we can find $s\in\dom(q_\beta(\gamma))$ with $\height{s}{\TTT^\gamma}=\height{t_0}{\TTT^\gamma}$ and $s<_{\TTT^\gamma} u$.  But then $s=t\in\dom(q_\beta(\gamma))$, a contradiction.  This shows that $\height{t}{\TTT^\gamma}\geq\alpha>\xi$ and hence there is a $k<n_\gamma$ with $t^{\alpha,\gamma}_k\leq_{\TTT^\gamma} t$.  But then $\height{u}{\TTT^\gamma}>\height{t}{\TTT^\gamma}>\xi=\rho(\beta)$ implies that $\height{u}{\TTT^\gamma}\geq\beta$ and hence there is an $l<n_\gamma$ with $t^{\beta,\gamma}_l\leq_{\TTT^\gamma} u$.  Since $t<_{\TTT^\gamma}u$ and $\height{\iota_{\alpha,\gamma}(k)}{\TTT^\gamma}=\xi=\height{\iota_{\beta,\gamma}(l)}{\TTT^\gamma}$, we know that $\iota_{\alpha,\gamma}(k)=\iota_{\beta,\gamma}(l)$ and therefore $\iota_{\beta,\gamma}(l)<_{\TTT^\gamma} t\in \beta\times\beta$ implies that $\langle\gamma,l\rangle\in K$. This shows that $\iota_{\alpha,\gamma}(k)=\iota_{\beta,\gamma}(l)=\iota_{\alpha,\gamma}(l)$ and the injectivity of $\iota_{\alpha,\gamma}$ implies that $k=l$ and $\langle\gamma,k\rangle\in K$. In this situation, the above choices ensure that $r_\gamma(t^{\alpha,\gamma}_k)=r_\gamma(t^{\beta,\gamma}_k)$ and $c^\gamma_{r_\gamma(t^{\alpha,\gamma}_k)}(t^{\alpha,\gamma}_k)=c^\gamma_{r_\gamma(t^{\alpha,\gamma}_k)}(t^{\beta,\gamma}_k)$. Since $t^{\alpha,\gamma}_k\neq t^{\beta,\gamma}_k$, this implies that $\height{\TTT^\gamma}{}=\kappa$ and hence the nodes $t^{\alpha,\gamma}_k$ and $t^{\beta,\kappa}_k$ are incompatible in $\TTT^\gamma$. But this yields a contradiction, because we have $t^{\alpha,\gamma}_k\leq_{\TTT^\gamma} t<_{\TTT^\gamma} u$ and $t^{\beta,\gamma}_k\leq_{\TTT^\gamma}u$. 

 The above computations show that the sequence $\seq{p_\alpha}{\alpha\in E}$ consists of pairwise compatible conditions in $\prod_{\gamma<\mu}\PPP(\TTT^\gamma)$. 
\end{proof}

We now introduce the $\kappa$-Knaster partial order that is used in the proofs of Theorem \ref{square_productivity_thm} and \ref{KnasterLayered_thm}.

\begin{definition}\label{definition:IndSquareKnasterPO}
 Suppose that $\vec{\calC}$ is a $\square^{\mathrm{ind}}(\kappa, \lambda)$-sequence. We let $\PPP_{\vec{\calC}}$ denote the lottery sum of the sequence $\seq{\PPP(\TTT^{\vec{\calC}}_i)}{i<\lambda}$ of partial orders, i.e., conditions in $\PPP_{\vec{\calC}}$ are pairs $\langle p,i\rangle$ with $i<\lambda$ and $p\in\PPP(\TTT^{\vec{\calC}}_i)$ and, given $\langle p,i\rangle,\langle q,j\rangle\in\PPP_{\vec{\calC}}$, we have $\langle p,i\rangle\leq_{\PPP_{\vec{\calC}}}\langle q,j\rangle$ if either $i=j$ and $p\leq_{\PPP(\TTT^{\vec{\calC}}_i)}q$ or $q=\mathbbm{1}_{\PPP(\TTT^{\vec{\calC}}_j)}$. 
\end{definition}

\begin{lemma} \label{lottery_sum_lemma} 
  Let $\vec{\calC}=\seq{C_{\alpha, i}}{\alpha < \kappa, ~ i(\alpha) \leq i < \lambda}$ 
 be a $\square^{\mathrm{ind}}(\kappa, \lambda)$-sequence with the property that the set $E^\kappa_{{\geq}\lambda}\cap S^{\vec{\calC}}_{{\geq}i}$ is stationary in $\kappa$ for all $i<\lambda$. 
  \begin{enumerate}
    \item If $\mu < \lambda$ is a cardinal with $\nu^\mu < \kappa$ for all $\nu < \kappa$, then the full support product  $\PPP_{\vec{\calC}}^\mu$ is $\kappa$-Knaster. 

    \item The full support product $\PPP_{\vec{\calC}}^\lambda$ does not satisfy the $\kappa$-chain condition. 
  \end{enumerate}
\end{lemma}

\begin{proof}
  (1) Let $\seq{p_\alpha}{\alpha<\kappa}$ be a sequence of conditions in $\PPP_{\vec{\calC}}^\mu$. 
   Since $\lambda^\mu < \kappa$ holds, we may assume that there is a function $\map{f}{\mu}{\lambda}$ and a sequence $\seq{q_\alpha}{\alpha<\kappa}$ of conditions in the full support product $\prod_{\gamma<\mu}\PPP(\TTT^{\vec{\calC}}_{f(\gamma)})$ with the property that $p_\alpha(\gamma)=\langle q_\alpha(\gamma),f(\gamma)\rangle$ holds for all $\alpha<\kappa$ and $\gamma<\mu$. Set $i_* = \lub(\ran{f})<\lambda$.  
   In this situation, Lemma \ref{lemma:TreeIndexedSquareSpecialSubset} shows that for every $\gamma<\mu$ with the property that the tree $\TTT_{f(\gamma)}^{\vec{\calC}}$ has height $\kappa$, the set $S^{\vec{\calC}}_{{\geq}i_*}\subseteq S^{\vec{\calC}}_{{>}f(\gamma)}$ is non-stationary with respect to $\TTT_{f(\gamma)}^{\vec{\calC}}$. 
   Since our assumptions imply that the set $E^\kappa_{{\geq}\lambda}\cap S^{\vec{\calC}}_{{\geq}i_*}$ is stationary in $\kappa$, we can apply Lemma \ref{knaster_lemma} to conclude that the product $\prod_{\gamma<\mu}\PPP(\TTT^{\vec{\calC}}_{f(\gamma)})$ is $\kappa$-Knaster. Hence there is an unbounded subset $U$ of $\kappa$ such that the sequence $\seq{q_\alpha}{\alpha\in U}$ consists of pairwise compatible conditions in $\prod_{\gamma<\mu}\PPP(\TTT^{\vec{\calC}}_{f(\gamma)})$ and this implies that the sequence $\seq{p_\alpha}{\alpha\in U}$ consists of pairwise compatible conditions in $\PPP_{\vec{\calC}}^\mu$.

  (2) Given $\alpha\in\acc(\kappa)$ and $i<\lambda$, the function $\{\langle\alpha,0\rangle\}$ is a condition in the partial order $\PPP(\TTT^{\vec{\calC}}_{\max\{i,i(\alpha)\}})$. This shows that for every $\alpha\in\acc(\kappa)$, there is a  unique condition $p_\alpha$ in $\PPP_{\vec{\calC}}^\lambda$ with $p_\alpha(i)=\langle\{\langle\alpha,0\rangle\},\max\{i,i(\alpha)\}\rangle$ for all $i<\lambda$. 
  Fix $\alpha,\beta\in\acc(\kappa)$ with $\alpha<\beta$. 
  Then there is $i(\beta)\leq i < \lambda$ such that $\alpha \in \acc(C_{\beta,i})$. This implies that  $i\geq i(\alpha)$, $C_{\alpha,i}=C_{\beta,i}\cap\alpha$ and $\alpha<_{\TTT^{\vec{\calC}}_i}\beta$. We can conclude that the conditions $p_\alpha(i)=\langle\{\langle\alpha,0\rangle\},i\rangle$ and $p_\beta(i)=\langle\{\langle\beta,0\rangle\},i\rangle$ are incompatible in $\PPP_\calC$ and therefore the condition $p_\alpha$ and $p_\beta$ are incompatible in $\PPP_\calC^\lambda$.  These computations show that the sequence $\seq{p_\alpha}{\alpha\in\acc(\kappa)}$ enumerates an antichain in $\PPP_{\vec{\calC}}^\lambda$.  
\end{proof}

The statement of Theorem \ref{square_productivity_thm} now follows directly from an application of Theorem \ref{square_building_thm} with $S=E^\kappa_{{\geq}\lambda}$ and Lemma \ref{lottery_sum_lemma}. Moreover, by combining the above with the results of \cite{MR3620068}, we can show that $\square(\kappa)$ implies the existence of a $\kappa$-Knaster partial order that is not $\kappa$-stationarily layered.

\begin{proof}[Proof of Theorem \ref{KnasterLayered_thm}]
 Assume that $\kappa$ is an uncountable regular cardinal with the property that every $\kappa$-Knaster partial order is $\kappa$-stationarily layered. Then {\cite[Theorem 1.11]{MR3620068}} shows that $\kappa$ is a Mahlo cardinal with the property that every stationary subset of $\kappa$ reflects. Assume, towards a contradiction, that $\square(\kappa)$ holds.  In this situation, we can apply Theorem \ref{square_building_thm} to obtain a $\square^{\mathrm{ind}}(\kappa,\aleph_0)$-sequence $\vec{\calC}=\seq{C_{\alpha, i}}{\alpha < \kappa, ~ i(\alpha) \leq i < \omega}$ with the property that there exists a $\square(\kappa)$-sequence $\seq{D_\alpha}{\alpha<\kappa}$ such that $\acc(D_\alpha)\subseteq\acc(C_{\alpha,i(\alpha)})$ holds for all $\alpha\in\acc(\kappa)$. 
 
 \begin{claim*}
  The set $\Set{\alpha\in\acc(\kappa)}{\otp(D_\alpha)<\alpha}$ is not stationary in $\kappa$. 
 \end{claim*} 
 
 \begin{proof}[Proof of the Claim]
  Assume, for a contradiction, that the set is stationary in $\kappa$. Then a pressing down argument yields $\xi<\kappa$ such that the set $$E ~ = ~ \Set{\alpha\in\acc(\kappa)}{\otp(D_\alpha)=\xi}$$ is stationary in $\kappa$. By the above remarks, there is an $\alpha<\kappa$ such that $\cof(\alpha)>\omega$ and the set $E\cap\alpha$ is stationary in $\alpha$. But then we can find $\beta,\gamma\in\acc(D_\alpha)\cap E$ with $\gamma<\beta$. This implies that $D_\gamma=D_\beta\cap\gamma$ and hence $\xi=\otp(D_\gamma)<\otp(D_\beta)=\xi$, a contradiction. 
 \end{proof}
 
 By the above claim, we can find a club $C$ in $\kappa$ consisting of strong limit cardinals such that $\otp(D_\alpha)=\alpha$ holds for all $\alpha\in C$. Let $\PPP_{\vec{\calC}}$ denote the partial order defined in Definition \ref{definition:IndSquareKnasterPO}. Then a combination of Theorem \ref{square_building_thm} with Lemma \ref{lottery_sum_lemma} implies that $\PPP_{\vec{\calC}}$ is $\kappa$-Knaster and, by our assumption, this shows that $\PPP_{\vec{\calC}}$ is $\kappa$-stationarily layered. Pick a sufficiently large regular cardinal $\theta>\kappa$. Then {\cite[Lemma 2.3]{MR3620068}} shows that there is an elementary substructure $M$ of $\HH{\theta}$ of cardinality less than $\kappa$ and $\alpha\in C$ such that $\alpha=\kappa\cap M$, $\vec{\calC}\in M$, and $\PPP_{\vec{\calC}}\cap M$ is a regular suborder of $\PPP_{\vec{\calC}}$. Set $p_0=\{\langle\alpha,0\rangle\}$. Then $p=\langle p_0,i(\alpha)\rangle$ is a condition $\PPP_{\vec{\calC}}$ and there is a reduct $q$ of $p$ in $\PPP_{\vec{\calC}}\cap M$, i.e., $q\in M$ is a condition in $\PPP_{\vec{\calC}}$ with the property that for every $r\in\PPP_{\vec{\calC}}\cap M$ with $r\leq_{\PPP_{\vec{\calC}}}q$, the conditions $p$ and $r$ are compatible in $\PPP_{\vec{\calC}}$. Then there is a condition $q_0$ in $\PPP(\TTT^{\vec{\calC}}_{i(\alpha)})\cap M$ with $q=\langle q_0,i(\alpha)\rangle$. Since the conditions $p_0$ and $q_0$ are compatible in $\PPP(\TTT^{\vec{\calC}}_{i(\alpha)})$, we know that $q_0(\beta)\neq 0$ holds for all $\beta\in\dom(q_0)$ with $\beta<_{\TTT^{\vec{\calC}}_{i(\alpha)}}\alpha$. Since $\alpha=\otp(D_\alpha)$ is a cardinal and $\dom(q_0)$ is a finite subset of $\alpha$, there is a $\gamma\in\acc(D_\alpha)$ with $\dom(q_0)\subseteq\gamma$. Then $\gamma\in\acc(C_{\alpha,i(\alpha)})$, $i(\gamma)\leq i(\alpha)$ and $\gamma<_{\TTT^{\vec{\calC}}_{i(\alpha)}}\alpha$. Moreover, the above remarks show that $r=\langle q_0\cup\{\langle\gamma,0\rangle\},i(\alpha)\rangle\in M$ is a condition in $\PPP_{\vec{\calC}}$ that strengthens $q$. But this implies that the conditions $p$ and $r$ are compatible in $\PPP_{\vec{\calC}}$, a contradiction. 
\end{proof}

  The proof of the following result is similar to that of Theorem \ref{all_ascent_path_thm} presented in Section \ref{ConsResultsTrees}.

\begin{theorem}
  Let $\kappa$ be an inaccessible cardinal with $\mathbbm{1}_{\Add{\kappa}{1}}\Vdash$\anf{\textit{$\check{\kappa}$ is weakly compact}}. If   $\lambda < \kappa$ is an infinite, regular cardinal, then the following statements hold in  a cofinality-preserving forcing extension of the ground model:
  \begin{enumerate}
    \item There is a $\kappa$-Knaster partial order $\PPP$ such that $\PPP^\mu$ is $\kappa$-Knaster for all 
      $\mu < \lambda$, but $\PPP^\lambda$ does not satisfy the $\kappa$-chain condition. 
      
    \item If $\RRR$ is a partial order with the property that $\RRR^\lambda$ satisfies the $\kappa$-chain condition, then $\RRR^\theta$ satisfies the $\kappa$-chain condition for all   $\theta < \kappa$. 
  \end{enumerate}
\end{theorem}

\begin{proof}
 Let $\PPP = \PPP(\kappa, \lambda)$ be the forcing notion from Definition \ref{ind_square_forcing_def} that adds a $\square^{\mathrm{ind}}(\kappa, \lambda)$-sequence, let $\dot{\calC}$ be a $\PPP$-name for the generically-added $\square^{\mathrm{ind}}(\kappa, \lambda)$-sequence, and, for all $i < \lambda$, let $\dot{\QQQ}_i$ be a $\PPP$-name for the partial order $\TTT_i(\dot{\calC})$ defined in Definition \ref{definition:POthreadIndexedSquare}.  If $i < \lambda$, then  our assumptions imply that $\kappa$ is weakly compact in all $(\PPP * \dot{\QQQ}_i)$-generic extensions of $\VV$.

  Let $G$ be $\PPP$-generic over $V$, and let $\dot{\calC}^G = \vec{\calC} = \seq{C_{\alpha, i}}{\alpha < \kappa, ~ i(\alpha) \leq i < \lambda}$   be the realization of $\dot{\calC}$. Given $i < j < \lambda$, set $\QQQ_i=\dot{\QQQ}_i^G$ and let $\map{\pi_{i,j}}{\QQQ_i}{\QQQ_j}$ be the projection map given by $\pi_{i,j}(C_{\alpha, i})=C_{\alpha, j}$. We claim that $\VV[G]$ is the desired forcing extension. 
  By Lemma \ref{lemma:PropertiesGenericIndexedSquare}, the set  $E^\kappa_{{\geq}\lambda}\cap S^{\vec{\calC}}_{{\geq}i}$ is a stationary subset of $\kappa$ in $\VV[G]$ for all $i<\lambda$. In this situation, Lemma \ref{lottery_sum_lemma} shows that the partial order $\PPP_{\vec{\calC}}$ from Definition \ref{definition:IndSquareKnasterPO} witnesses that the above statement (1) holds. 

  Let us now show that requirement (2) holds in $\VV[G]$. To this end, work in $\VV[G]$ and fix a partial order $\RRR$ such that $\RRR^\lambda$ satisfies   the $\kappa$-chain condition. 

  \begin{claim*}
    There is an $i < \lambda$ and a condition $q$ in $\QQQ_i$ such that $$q \Vdash_{\QQQ_i}\anf{\textit{$\check{\RRR}$ satisfies the $\check{\kappa}$-chain condition}}.$$
  \end{claim*}

  \begin{proof}[Proof of the Claim]
    Suppose not. Given $i < \lambda$, this assumption yields  a sequence $\seq{\dot{r}_{i, \eta}}{\eta < \kappa}$     of $\QQQ_i$-names for elements of $\bb{R}$ such that $$\mathbbm{1}_{\QQQ_i}\Vdash\anf{\textit{The conditions $\dot{r}_{i,\eta}$ and $\dot{r}_{i,\xi}$ are incompatible in $\check{\RRR}$}}$$ for all $\eta<\xi<\kappa$.     For each $\eta < \kappa$, use Lemma \ref{deciding_lemma} to find $q_\eta \in \QQQ_0$  and a sequence $\seq{r_{i, \eta}}{i<\lambda}$ of conditions in $\RRR$  such that $\pi_{0,i}(q_\eta)\Vdash_{\QQQ_i}\anf{\dot{r}_{i,\eta}=\check{r}_{i,\eta}}$  for all $i < \lambda$. Given $\eta < \kappa$, pick $\alpha_\eta \in \acc(\kappa)$  with $q_\eta = C_{\alpha_\eta, 0}$. 

    For each $\eta < \lambda$, let $s_\eta$ denote the unique condition in $\RRR^\lambda$ with $s_\eta(i) = r_{i, \eta}$ for all $i < \lambda$. 
    Fix $\eta,\xi < \lambda$ with $\alpha_\eta < \alpha_\xi$ and $i < \lambda$ with $\alpha_\eta \in \acc(C_{\alpha_\xi,j})$ for all $i \leq j < \lambda$. Given $i \leq j < \lambda$, we then have $\pi_{0,i}(q_\xi) \leq_{\QQQ_i} \pi_{0,i}(q_\eta)$ and therefore $$\pi_{0,i}(q_\xi) \Vdash_{\QQQ_i}\anf{\textit{$\dot{r}_{i, \eta} = \check{r}_{i,\eta}$  and  $\dot{r}_{i, \xi} = \check{r}_{i,\xi}$}}.$$   In particular, the conditions $s_\eta(i)$ and $s_\xi(i)$ are incompatible in $\RRR$, and therefore the conditions  $s_\eta$ and $s_\xi$ are incompatible     in $\RRR^\lambda$.     But this shows that $\Set{s_\eta}{\eta < \kappa}$ is an antichain in $\RRR^\lambda$ of size $\kappa$,     contradicting our assumption that $\RRR^\lambda$ satisfies the $\kappa$-chain condition. 
  \end{proof}

  Fix $i$ and $q$ as given in the claim, and $\theta<\kappa$. Let $H$ be $\QQQ_i$-generic over $\VV[G]$ with $q \in H$. In $\VV[G,H]$, $\kappa$ is weakly compact  and $\RRR$ satisfies the $\kappa$-chain condition. By the weak  compactness of $\kappa$, it follows that $\RRR^\theta$ is $\kappa$-Knaster in $\VV[G,H]$. Since $\QQQ_i$ is ${<}\kappa$-distributive in $\VV[G]$, we have $(\RRR^\theta)^{\VV[G]}=(\RRR^\theta)^{\VV[G,H]}$. Moreover, since the property of satisfying the $\kappa$-chain condition is easily seen   to be downward absolute, we can conclude that $\RRR^\theta$ satisfies the $\kappa$-chain condition   in $\VV[G]$.
\end{proof}



\bibliographystyle{amsplain}
\bibliography{ascent_paths}


\end{document}